\documentclass{amsart}

\usepackage{amssymb}
\usepackage[all]{xy}
\usepackage{caption}
\usepackage[pdftex]{hyperref}
\usepackage{graphicx}

\newtheorem{thm}{Theorem}
\newtheorem{lem}[thm]{Lemma}
\newtheorem{prop}[thm]{Proposition}
\newtheorem{cor}[thm]{Corollary}

\theoremstyle{definition}

\newcommand{\pgl}{\PGL(2,\mathbb{C})}
\newcommand{\zz}{ \mathbb{Z}_{2}\times \mathbb{Z}_{2}}
\newcommand{\sldos}{\SL(2,\mathbb{C})}
\DeclareMathOperator{\tr}{tr}
\DeclareMathOperator{\Id}{Id}

\DeclareMathOperator{\GL}{GL}

\DeclareMathOperator{\SL}{SL}

\DeclareMathOperator{\PGL}{PGL}

\DeclareMathOperator{\Gr}{Gr}

\DeclareMathOperator{\Tr}{Tr\,}       

\newcommand{\cD}{\mathcal{D}}
\newcommand{\cM}{\mathcal{M}} 

\newcommand{\cH}{\mathcal{H}}
\newcommand{\x}{\times}
\newcommand{\CC}{\mathbb{C}} 
\newcommand{\PP}{\mathbb{P}} 
\newcommand{\RR}{\mathbb{R}} 
\newcommand{\ZZ}{\mathbb{Z}} 
\newcommand{\too}{\longrightarrow}
\newcommand{\imat}{\sqrt{-1}} 
\newcommand{\surj}{\twoheadrightarrow}
\newcommand{\im}{\mathrm{im}\,}

\author[J. Mart\'{\i}nez]{Javier Mart\'{\i}nez}
\address{Facultad de Matem\'aticas, Universidad Complutense de Madrid,
Plaza Ciencias 3, 28040 Madrid, Spain}
\email{javiermartinez@mat.ucm.es}

\author[V. Mu\~{n}oz]{Vicente Mu\~{n}oz}
\address{Facultad de Matem\'aticas, Universidad Complutense de Madrid,
Plaza Ciencias 3, 28040 Madrid, Spain}
\address{Instituto de Ciencias Matem\'aticas (CSIC-UAM-UC3M-UCM),
C/ Nicolas Cabrera 15, 28049 Madrid, Spain}
\email{vicente.munoz@mat.ucm.es}

\title[E-polynomials of $\SL(2,\CC)$-character varieties]{E-polynomials 
of $\SL(2,\mathbb{C})$-character varieties \\ of complex curves of genus $3$}

\date{15 May 2014}



\begin{document}

\begin{abstract}
We compute the E-polynomials of the moduli spaces of representations of the
fundamental group of a complex curve of genus $g=3$ into $\SL(2,\CC)$,
and also of the moduli space of twisted representations. The case
of genus $g=1,2$ has already been done in \cite{lomune}. 
We follow the geometric technique introduced in \cite{lomune}, based on stratifying the space of
representations, and on the analysis of the behaviour of the E-polynomial under fibrations.
\end{abstract}

\maketitle

\section{Introduction}\label{sec:introduction}

Let $X$ be a smooth complex projective curve of genus $g\geq 1$, and let $G$ be a complex reductive group.
The $G$-character variety  of $X$ is defined as the moduli space of semisimple representations of
$\pi_{1}(X)$ into $G$, that is,
 $$
  \cM (G)= \{(A_{1},B_{1},\ldots,A_{g},B_{g}) \in G^{2g} \,| \prod_{i=1}^{g}[A_{i},B_{i}]=\Id\}/ / G.
 $$
For complex linear groups $G=\GL(n,\CC), \SL(n,\CC)$, the representations of $\pi_1(X)$ into $G$ can be understood
as $G$-local systems $E\to X$, hence defining
a flat bundle $E$ which has $\deg E=0$. A natural generalization consists of allowing bundles $E$ of
non-zero degree $d$. The $G$-local systems on $X$ correspond to representations $\rho:\pi_1(X - \{p_0\}) \to G$,
where $p_0\in X$ is a fixed point, and $\rho(\gamma)= \frac{d}{n} \Id$, $\gamma$ a loop around $p_0$, giving
rise to the moduli space of twisted representations
 $$
  \cM^d (G)= \{(A_{1},B_{1},\ldots,A_{g},B_{g}) \in G^{2g} \,| \prod_{i=1}^{g}[A_{i},B_{i}]=e^{2\pi \imat d/n} \Id \}/ / G.
 $$

The space $\cM^d(G)$ is known in the literature as the Betti moduli space.
This space is closely related to two other spaces: the De Rham and Dolbeault moduli spaces. The De Rham moduli space
$\cD^d(G)$ is the moduli space parameterizing flat bundles, i.e., $(E,\nabla)$, where $E\to X$ is an algebraic bundle of
degree $d$ and rank $n$ (and fixed determinant in the case $G=\SL(n,\CC)$), and
$\nabla$ is an algebraic connection on $X- \{p_0\}$ with a logarithmic pole at
$p_0$ with residue $-\frac{d}{n} \Id$. The Riemann-Hilbert correspondence
\cite{deligne:1970,simpson:1995} gives an analytic isomorphism (but not an algebraic isomorphism)
$\cD^d(G) \cong \cM^d(G)$.
The Dolbeault moduli space $\cH^d(G)$ is the moduli space of $G$-Higgs
bundles $(E,\Phi)$, consisting of an algebraic bundle
$E\to X$ of degree $d$ and rank $n$ (also with fixed determinant in the case $G=\SL(n,\CC)$),
and a homomorphism $\Phi:E\to E\otimes K_X$, known as the Higgs field (in the case $G=\SL(n,\CC)$, 
the Higgs field has trace $0$). In this situation the
theory of harmonic bundles \cite{corlette:1988,simpson:1992} gives a homeomorphism $\cM^d(G)\cong \cH^d(G)$.

When $\gcd(n,d)=1$ these moduli spaces are smooth and the underlying differentiable manifold is
hyperk\"ahler, but the complex structures do not correspond under the previous isomorphisms.
The cohomology of these moduli spaces has been computed in several particular cases, mostly
from the point of view of the Dolbeault moduli space $\cH^{d}(G)$.
Poincar\'e polynomials for $G=\SL(2,\CC)$ were computed by Hitchin in his seminal paper on Higgs
bundles \cite{hitchin:1987} and for $G=\SL(3,\CC)$ by Gothen in \cite{gothen:1994}.
More recently the case of $G=\GL(4,\CC)$ has been solved in 
\cite{garciaprada-heinloth-schmitt:2011}, and a recursive formula for the motive of the moduli space of Higgs bundles of arbitrary rank and degree coprime to the rank has been given in \cite{oscar.}. In particular, this gives the Betti numbers of the character variety for arbitrary coprime rank and degree.

Hausel and Rodriguez-Villegas started the computation of the E-polynomials of $G$-character
varieties focusing on $G=\GL(n,\CC), \SL(n,\CC)$ and $\PGL(n,\CC)$ using arithmetic methods inspired by the Weil
conjectures. In \cite{hausel-rvillegas:2007} they obtained the E-polynomials of $\cM^{d}(G)$
for $G=\GL(n,\CC)$ in terms of a simple generating function.
Following these methods, Mereb \cite{mereb:2010} studied the case of $\SL(n,\CC)$ giving
an explicit formula for the E-polynomial in the case $G=\SL(2,\CC)$, while for $\SL(n,\CC)$ 
these polynomials are given in terms of a generating function. He proved also that the E-polynomials 
for $\SL(n,\CC)$ are palindromic and monic. 

Logares, Newstead and the second author \cite{lomune} introduced a geometric technique
to compute the E-polynomial of character varieties 
by using stratifications and also handling fibrations
which are locally trivial in the analytic topology but not in the Zariski topology. The
main results of \cite{lomune} are explicit formulae for the E-polynomials
for character varieties for $G=\SL(2,\CC)$ and $g=1,2$. Actually, the geometric technique
allows for dealing with ``twisted'' character varieties in which the holonomy around the puncture
is not diagonizable (in this case there is no correspondence with 
a De Rham or Dolbeault moduli space as mentioned above).

The purpose in this paper is to give another step in the geometric technique of \cite{lomune}
to the case of $g=3$. We prove the following

\begin{thm}\label{thm:main}
  Let $X$ be a complex curve of genus $g=3$. Then the E-polynomials of the character varieties are:
  \begin{align*}
    e(\cM(\SL(2,\CC))) &= q^{12}-4q^{10}+74q^8+375q^6+16q^4+q^2+1\, , \\
    e(\cM^1(\SL(2,\CC))) &= q^{12}-4q^{10}+6q^8-252q^7-14q^6-252q^5+6q^4-4q^2+1 \, ,
  \end{align*}
where $q=\, uv$. 
\end{thm}

The new input consists of using the map $\cM(\SL(2,\CC)) \to \CC^3$ given by $(A_1,B_1,A_2,B_2,A_3,B_3)\mapsto
(t_1,t_2,t_3)$, $t_i=\tr ([A_i,B_i])$, to stratify conveniently the space of representations
in $\SL(2,\CC)^6$. This leads us to analyse a fibration over a two-dimensional basis (the
quadric $C$ in Figure \ref{fig1}), for which 
we have to compute the E-polynomial. Therefore the arguments of \cite{lomune}
about Hodge monodromy fibrations over one-dimensional bases have to be extended to 
general bases (see Theorem \ref{thm:general-fibr}). 
In the case of $\cM(\SL(2,\CC))$ we also have to deal with reducible representations. The arguments
used in \cite[Section 8.2.2]{lomune} to compute the E-polynomial of the irreducible locus 
give an increasingly large number of strata as the genus increases. So here we use an
alternative route (suggested by the ideas in \cite{lamu}), consisting of computing
the E-polynomial of the irreducible locus by subtracting the E-polynomial of 
the subset of reducible representations.

As a byproduct of our analysis, we obtain the behaviour of the E-polynomial of the parabolic character
variety ($G=\SL(2,\CC)$)
 $$
  \cM^\lambda (G)= \left\{(A_{1},B_{1},\ldots,A_{g},B_{g}) \in G^{2g} \,| \prod_{i=1}^{g}[A_{i},B_{i}]=
  \begin{pmatrix} \lambda & 0 \\ 0 &\lambda^{-1} \end{pmatrix} \right\}/ / G.
 $$
when $\lambda$ varies in $\CC-\{0,\pm 1\}$, for the case $g=2$. This is given by the following formula
  \begin{equation}\label{eqn:RcM}
  R(\cM^\lambda)=(q^8+q^7-2q^6-2q^5+4q^4-2q^3-2q^2+q+1)T+ 15(q^5-2q^4+q^3)N.
  \end{equation}
which means that the E-polynomial of the invariant part of the cohomology is the polynomial accompanying $T$,
and the  E-polynomial of the non-invariant part is the polynomial accompanying $N$.  This monodromy is of interest
from the point of view of Mirror Symmetry \cite{hausel-thaddeus:2003} and answers a question raised to the authors by Tam\'as Hausel.

After finishing and submitting a first version of this paper, the authors have managed to complete the computation of E-polynomials of character varieties for surfaces of general genus $g \geq 3$ in \cite{mamu2}. The case provided here cannot however be subsumed in \cite{mamu2}, since it serves as the starting point of induction on the genus carried there. Moreover, this specific case of $g=3$ has special features and has to be treated by an \textit{ad-hoc} method (it is the only case that needs the analysis of Hodge monodromy representations over a $2$-dimensional basis).

\noindent\textbf{Acknowledgements.} 
We would like to thank Peter Newstead, Tam\'as Hausel, Martin Mereb, Nigel Hitchin, Marina Logares and Oscar Garc\'{\i}a-Prada for useful conversations.
This work has been partially supported by MICINN (Spain) Project MTM2010-17389. The first author was also supported by a FPU scholarship from the Spanish Ministerio de Educaci\'on.

\section{E-polynomials}\label{sec:method}

Our main goal is to compute the Hodge-Deligne polynomial of $\SL(2,\CC)$-character varieties. 
We will follow the methods in \cite{lomune}, so we collect some basic results from \cite{lomune} in this section.

\subsection{E-polynomials}
We start by reviewing the definition of E-polynomials. 
A pure Hodge structure of weight $k$ consists of a finite dimensional complex vector space
$H$ with a real structure, and a decomposition $H=\bigoplus_{k=p+q} H^{p,q}$
such that $H^{q,p}=\overline{H^{p,q}}$, the bar meaning complex conjugation on $H$.
A Hodge structure of weight $k$ gives rise to the so-called Hodge filtration, which is a descending filtration
$F^{p}=\bigoplus_{s\ge p}H^{s,k-s}$. We define $\Gr^{p}_{F}(H):=F^{p}/ F^{p+1}=H^{p,k-p}$.

A mixed Hodge structure consists of a finite dimensional complex vector space $H$ with a real structure,
an ascending (weight) filtration $\ldots \subset W_{k-1}\subset W_k \subset \ldots \subset H$
(defined over $\RR$) and a descending (Hodge) filtration $F$ such that $F$ induces a pure Hodge structure
of weight $k$ on each $\Gr^{W}_{k}(H)=W_{k}/W_{k-1}$. We define
 $$
 H^{p,q}:= \Gr^{p}_{F}\Gr^{W}_{p+q}(H)
 $$
and write $h^{p,q}$ for the {\em Hodge number} $h^{p,q} :=\dim H^{p,q}$.

Let $Z$ be any quasi-projective algebraic variety (maybe non-smooth or non-compact). 
The cohomology groups $H^k(Z)$ and the cohomology groups with compact support  
$H^k_c(Z)$ are endowed with mixed Hodge structures \cite{Deligne2,Deligne3}. 
We define the {\em Hodge numbers} of $Z$ by
 $$
  h^{k,p,q}_{c}(Z) = h^{p,q}(H_{c}^k(Z))=\dim \Gr^{p}_{F}\Gr^{W}_{p+q}H^{k}_{c}(Z) .
 $$
The Hodge-Deligne polynomial, or E-polynomial is defined as 
 $$
 e(Z)=e(Z)(u,v):=\sum _{p,q,k} (-1)^{k}h^{k,p,q}_{c}(Z) u^{p}v^{q}.
 $$

When $h_c^{k,p,q}=0$ for $p\neq q$, the polynomial $e(Z)$ depends only on the product $uv$.
This will happen in all the cases that we shall investigate here. In this situation, it is
conventional to use the variable $q=uv$. If this happens, we say that the variety is {\it of balanced type}.
For instance, $e(\CC^n)=q^n$.

The key property of Hodge-Deligne polynomials that permits their calculation is that they are additive for
stratifications of $Z$. If $Z$ is a complex algebraic variety and
$Z=\bigsqcup_{i=1}^{n}Z_{i}$, where all $Z_i$ are locally closed in $Z$, then $e(Z)=\sum_{i=1}^{n}e(Z_{i})$.

There is another useful property that we shall use often: if there is an action of $\ZZ_2$ on $X$, then we have
polynomials $e(X)^+, e(X)^-$, which are the E-polynomials of the invariant and anti-invariant parts of the cohomology
of $X$, respectively. More concretely, $e(X)^+=e(X/\ZZ_2)$ and $e(X)^-=e(X)-e(X)^+$. Then if $\ZZ_2$ acts
on $X$ and on $Y$, we have the equality (see \cite[Proposition 2.6]{lomune}) 
 \begin{equation}\label{eqn:+-}
e((X\x Y)/\ZZ_2) =e(X)^+e(Y)^+ +e(X)^-e(Y)^-\, .
  \end{equation}

\subsection{Fibrations}
We have to deal with fibrations
 \begin{equation}\label{fibration}
  F \longrightarrow Z \overset{\pi}{\longrightarrow} B 
 \end{equation}
that are locally trivial in the analytic topology. We want to compute the E-polynomial of the total space in terms 
of the polynomials of the base and the fiber. The fibration defines a local system $\mathcal{H}^{k}_{c}$, 
whose fibers are the cohomology groups $H^{k}_{c}(F_{b})$, where $b\in B$, $F_{b}=\pi^{-1}(b)$. 
The fibers possess mixed Hodge structures, and the subspaces $W_{t}(H^{k}_{c}(F_{b}))$ are preserved 
by the holonomy. We will assume from now on that $F$ is of balanced type, so $\Gr_{2p}^W H^{k}_{c}(F_{b})=
H^{k,p,p}_{c}(F_{b})$. Associated to the fibration, there is a monodromy representation:
 \begin{equation}\label{eqn:exxtra}
 \rho : \pi_{1}(B) \longrightarrow  \GL(H^{k,p,p}_c(F)).
 \end{equation}
Suppose that the  monodromy group $\Gamma=\im(\rho)$ is an abelian and finite group. 
Then $H_c^{k,p,p}(F)$ are modules over the representation ring $R(\Gamma)$. So there is a well
defined element, the  {\em Hodge monodromy representation},
  \begin{equation}\label{eqn:Hodge-mon-rep}
  R(Z) := \sum (-1)^k H_c^{k,p,p}(F)\,  q^p \in R(\Gamma)[q] \, .
  \end{equation}

As the monodromy representation (\ref{eqn:exxtra}) has finite image, there is a finite covering
$B_\rho \to B$ such that the pull-back fibration
 \begin{equation}\label{eqn:fibration-bis}
 \xymatrix{Z' \ar[r] \ar[d]_{\pi'} & Z \ar[d]_\pi \\
 B_\rho \ar[r] & B}
 \end{equation}
has trivial monodromy. 

Let $S_1,\ldots, S_N$ be the
irreducible representations of $\Gamma$ (there are $N=\# \Gamma$ of them, and all
of them are one-dimensional). These are generators of $R(\Gamma)$ as a free abelian group. We
write the Hodge monodromy representation of (\ref{eqn:Hodge-mon-rep}) as
  $$
 R(Z)= a_1(q) S_1 + \ldots a_N(q) S_N.
 $$

\begin{thm} \label{thm:general-fibr}
 Suppose that $B_\rho$ is of balanced type. Then $Z$ is of balanced type. Moreover, there are 
 polynomials $s_1(q),\ldots, s_N(q)\in \ZZ[q]$ (only dependent on $B,B_\rho$ and $\Gamma$, but not on the
fibration or the fiber) such that 
 $$
 e(Z)= a_1(q) s_1(q) + \ldots +a_N(q) s_N(q),
 $$
for $R(Z)= a_1(q) S_1 + \ldots a_N(q) S_N$.
\end{thm}

\begin{proof}
 The Leray spectral sequence of the fibration (\ref{fibration}) has $E_2$-term
  \begin{equation}\label{eqn:E2-term}
    E_2^{l,m}(Z)=H^l_c(B,H^m_c(F))
  \end{equation}
 and abuts to $H^k_c(Z)$. By \cite{a}, $E_j^{l,m}(Z)$ has a mixed Hodge structure for
$j\geq 2$, and the differentials $d_j$ are compatible with the mixed Hodge structure.
Therefore $e(Z)=e(H^*_c(Z))=e(E_2^{*,*}(Z))$.

The mixed Hodge structure associated to $\pi'$ in (\ref{eqn:fibration-bis}) is
the product mixed Hodge structure $E_2^{l,m}(Z')=H^l_c(B_\rho)\otimes H^m_c(F)$. By our
assumption, $B_\rho$ and $F$ are of balanced type, so $Z'$ is of balanced type. There is
a map $E_2^{l,m}(Z') \to E_2^{l,m}(Z)$, which preserves the mixed Hodge structures.
This map is surjective, so $Z$ is of balanced type. 

By definition, 
 $$ 
 R(Z) = \sum (-1)^m H_c^{m,p,p}(F)\,  q^p = a_1(q) S_1 + \ldots + a_N(q) S_N\, .
 $$
The local systems $S_i \to B$ are $1$-dimensional 
and have a mixed Hodge structure. When we pull-back $S_i \to B$ to 
$B_\rho$, we get trivial local systems. Hence $H^l_c(B_\rho) \twoheadrightarrow H^l_c(B,S_i)$
and it has the induced mixed Hodge structure. So
$e(H^*_c(B,H^*_c(F)))= e(H^*_c(B, R(Z)))=\sum a_i(q) e(H^*_c(B,S_i))$.
We define $s_i(q)=e(H^*_c(B,S_i))$ to have the statement.
\end{proof}

Write $e(S_i)=s_i(q)$, $1\leq i\leq N$. So there is a $\ZZ[q]$-linear map
 $$
 e: R(\Gamma)[q] \to \ZZ[q]
 $$
satisfying the property that $e(R(Z))=e(Z)$.

\medskip

A trivial application of Theorem \ref{thm:general-fibr} happens when 
$\pi:Z\to B$ is a fibre bundle with fibre $F$ such that the action of $\pi_1(B)$ on $H_c^*(F)$ is trivial.
Then $R(Z)=e(F) T$, where $T$ is the trivial local system and
$e(Z)=e(F) e(B)$ (this result appears in \cite[Proposition 2.4]{lomune}).

The hypothesis that the action of $\pi_1(B)$ on $H_c^*(F)$ is trivial holds in particular in the following cases:
\begin{itemize}
\item $B$ is irreducible and $\pi$ is locally trivial in the Zariski topology. 
\item $\pi$ is a principal $G$-bundle with $G$ a connected algebraic group. 
\item $Z$ is a $G$-space with isotropy $H<G$ such that $G/H\to Z\to B$ is a fiber bundle, 
and $G$ is a connected algebraic group.
\end{itemize}

We can also recover easily the result in \cite[Proposition 2.10]{lomune}, which is the main
tool used in \cite{lomune}.

\begin{cor}\label{cor:general-fibr}
Let $B=\CC-\{q_1,\ldots, q_\ell\}$. Suppose that $B_\rho$ is a rational curve. Then 
 $e(Z) =(q-1)\, e(F)^{inv} - (\ell-1) e(F)$, where $e(F)^{inv}$ denotes the E-polynomial of
the invariant part of the cohomology of $F$.
\end{cor}

\begin{proof}
 Let $\Gamma$ be the monodromy, $N=\# \Gamma$, and $S_1=T, S_2,\ldots, S_N$
the irreducible representations, where $T$ denotes the trivial representation.
We only need to see that $e(T)=q-\ell$ and  $e(S_i) = - (\ell-1)$ for $i\geq 2$.
Given that, if $R(Z)=a T+ \sum_{i\geq 2} a_i S_i$, then $e(F)^{inv}=a$, $e(F)=a+\sum a_i$,
and $e(Z)=(q-\ell) a - (\ell-1) \sum a_i=(q-1)a -(\ell-1)(a+\sum a_i)$, as required.

For the trivial representation, it is clear that $e(T)=e(B)=q-\ell$. 

Now let $S$ be a (one-dimensional) irreducible representation of
$\Gamma$. 
Let $\tilde\Gamma$ be the image of $S:\Gamma\to \CC^*$, and let 
$e=\#\tilde\Gamma$. Take the $e$-cover asociated to this group, $\tilde B \to B$. Then 
$\tilde\Gamma$ acts on $\tilde B$ with quotient $B$. Clearly, $\tilde B$ is a rational curve
(the quotient $\Gamma \to \tilde\Gamma$ produces a covering map $B_\rho \to \tilde B$,
and $B_\rho$ is a rational curve by assumption). Then we have a fibration
$Y \to \tilde B\to B$, where $Y$ is a finite set of $e$ points. Clearly,
$R(\tilde B)= T+ \sum_{p=2}^e S_{i_p}$, for some $i_p\in \{2,\ldots,N\}$, where $S_{i_2}=S$.

The covering $\tilde B\to B$ can be extended to a ramified covering
$\varphi:\PP^1 \to \PP^1$. 
Hurwitz formula then says that $-2=e\, (-2) +r$, where $r$ is the degree of the
ramification divisor. Then $\varphi^{-1}(q_1,\ldots,q_\ell, \infty)$ has
$e(\ell+1) -r$ points. Therefore $\tilde B=\PP^1-\varphi^{-1}(q_1,\ldots,q_\ell, \infty)$
has 
 $$
 e(\tilde B)=q+1-e(\ell+1)+r=q+1-e(\ell+1) + (2e-2)=(q-\ell) - (e-1)(\ell-1).
 $$ 
The formula in Theorem \ref{thm:general-fibr} says that
$e(\tilde B)=(q-\ell) + \sum_{p=2}^e s_{i_p}$. Therefore
$\sum_{p=2}^e s_{i_p}=- (e-1)(\ell-1)$.

This happens for all choices of coverings associated to all representations $S_2,\ldots, S_N$.
Hence $s_i=-(\ell-1)$, for all $i=2,\ldots,N$.
\end{proof}

Another case that we are going to use later is the following.

\begin{cor} \label{cor:thm1} 
Let (\ref{fibration}) be a fibration over $B=\mathbb{C}^{\ast}\times \mathbb{C}^{\ast}$ with finite monodromy, 
such that $F$ is of balanced type. Then $Z$ is also of balanced type and its E-polynomial is given by
 $$
 e(Z)=(q-1)^2 e(F)^{inv}.
 $$
\end{cor}

\begin{proof}
First note that $\pi_1(B)=\ZZ\x \ZZ$, so the monodromy $\Gamma$, being a quotient 
$\ZZ\x\ZZ\surj \Gamma$, must be abelian. Moreover, there is some $n>0$ such that $\ZZ_n\x \ZZ_n\surj \Gamma$,
and the covering associated to $\ZZ_n\x \ZZ_n$ is of balanced type (is $\CC^*\x \CC^*$ again). Hence 
$B_\rho$ is of balanced type, since there are coverings $\CC^*\x \CC^* \to B_\rho \to
\CC^*\x \CC^*$. Now let $S$ be an irreducible representation of $R(\Gamma)$. If $S=T$, the trivial
representation, then $e(T)=e(B)=(q-1)^2$. 

If $S$ is a non-trivial representation, then take the covering $\tilde B\to B$ associated to $S$. Again 
there are coverings $\CC^*\x \CC^* \to B_\rho \to \tilde B \to B=\CC^*\x \CC^*$. 
The Hodge monodromy representation associated to $\tilde B$ is $T+ \sum_{p=2}^e S_{i_p}$, where $S_{i_2}=S$.
In cohomology $H^*_c(\CC^*\x \CC^*)\to H^*_c(\tilde B) \to
H^*_c(B)$ are surjections, but the composition is an isomorphism (multiplication by $n$).
Therefore $e(R(\tilde B))=e(\tilde B)=(q-1)^2$, and $e(\sum_{p=2}^e S_{i_p})=0$. This happens for 
all choices of $S$, so it must be $e(S)=0$, for any irreducible non-trivial $S$.
\end{proof}

\subsection{Building blocks}
We need to recall some E-polynomials from \cite{lomune} which serve as building blocks
for the E-polynomials of character varieties for any genus $g\geq 2$. They are basically
associated to the case of $g=1$. First, 
we have that $e(\SL(2,\CC))=q^3-q$ and $e(\PGL(2,\CC))=q^3-q$.
Consider the following subsets of $\SL(2,\CC)$:
\begin{itemize}
\item $W_{0}:=$ conjugacy class of $\left(
             \begin{array}{cc}
               1 & 0 \\
               0 & 1
             \end{array}
           \right)$. It has $e(W_0)=1$.
\item $W_{1}:=$ conjugacy class of  $\left(
             \begin{array}{cc}
               -1 & 0 \\
               0 & -1
             \end{array}
           \right)$. It has $e(W_1)=1$.
\item $W_{2}:=$ conjugacy class of $J_+=\left(
              \begin{array}{cc}
                1 & 1 \\
                0 & 1
              \end{array}
            \right)$. It is $W_2 \cong \PGL(2,\CC)/U$, with $U=\left\{ \left(
              \begin{array}{cc}
                1 & y \\
                0 & 1
              \end{array} \right) \,|\, y\in\CC \right\}$. It has $e(W_2)=q^2-1$.
\item $W_{3}:=$ conjugacy class of $J_-=\left(
              \begin{array}{cc}
                -1 & 1 \\
                0 & -1
              \end{array}
            \right)$. It is $W_3\cong \PGL(2,\CC)/U$ and  $e(W_3)=q^2-1$.
\item $W_{4,\lambda}:=$ conjugacy class of
  $\xi_\lambda=\left(
              \begin{array}{cc}
                \lambda & 0 \\
                0 & \lambda^{-1}
              \end{array}
            \right)$, where $\lambda\in\CC-\{0,\pm1\}$. Note that $W_{4,\lambda}=W_{4,\lambda^{-1}}$,
since the matrices $\xi_\lambda$ and $\xi_{\lambda^{-1}}$ are conjugated.
We have $W_{4,\lambda} \cong \PGL(2,\CC)/D$, where $D=\left\{ \left(
              \begin{array}{cc}
                x & 0 \\
                0 & x^{-1}
              \end{array} \right) \,|\, x\in\CC^* \right\}$. So $e(W_{4,\lambda})=q^2+q$.
\item We also need the set $W_{4}:=\{A\in\SL(2,\CC)\, | \, \Tr(A)\ne\pm2\}$,
which is the union of the conjugacy classes $W_{4,\lambda}$, $\lambda \in \CC- \{0,\pm 1\}$.
This has $e(W_4)= q^3-2q^2-q$.
\end{itemize}

Now consider the map
\begin{eqnarray*}
f:\SL(2,\CC)^{2}&\longrightarrow & \SL(2,\CC)\\
(A,B)&\mapsto& [A,B]=ABA^{-1}B^{-1}.
\end{eqnarray*}
Note that $f$ is equivariant under the action of $\SL(2,\CC)$ by conjugation on both spaces. 
We stratify $X=\SL(2,\CC)^{2}$ as follows:
\begin{itemize}
\item $X_{0}:=f^{-1}(W_{0})$,
\item $X_{1}:=f^{-1}(W_{1})$,
\item $X_{2}:=f^{-1}(W_{2})$,
\item $X_{3}:=f^{-1}(W_{3})$,
\item $X_{4}:=f^{-1}(W_{4})$.
\end{itemize}
We also introduce the varieties $f^{-1}(C)$ for fixed $C\in \SL(2,\CC)$ and define accordingly
 \begin{itemize}
\item $\overline{X}_{2}:=f^{-1}(J_+)$. There is a fibration $\PGL(2,\CC)/U \to X_2 \to \overline{X}_2$, so $e(X_2)=(q^2-1)e(\overline{X}_2)$.
\item $\overline{X}_{3}:=f^{-1}(J_-)$. There is a fibration $\PGL(2,\CC)/U \to X_3 \to \overline{X}_3$,
and $e(X_3)=(q^2-1)e(\overline{X}_3)$.
\item $\overline{X}_{4,\lambda}:=f^{-1}(\xi_\lambda)$, for $\lambda\ne 0,\pm1$.
We define also $X_{4,\lambda}=f^{-1}(W_{4,\lambda})$. There is a fibration 
$\PGL(2,\CC)/D \to X_{4,\lambda} \to \overline{X}_{4,\lambda}$,
and $e(X_{4,\lambda})=(q^2+q)e(\overline{X}_{4,\lambda})$.
\end{itemize}

It will also be convenient to define
\begin{itemize}
\item $\overline{X}_{4}:=\left\{(A,B,\lambda)\, | \, [A,B]=\left(\begin{array}{cc}
      \lambda & 0 \\
      0 & \lambda^{-1}\\
    \end{array}
  \right), \text{ for some } \lambda\ne0, \pm1,\, A,\,B \in \SL(2,\CC)\right\}$.
\end{itemize}
There is an action of $\ZZ_2$ on $\overline{X}_4$ given by interchanging 
$(A,B,\lambda)\mapsto (P^{-1}_0AP_0, P^{-1}_0BP_0,\lambda^{-1})$, with $P_0=
\left(\begin{array}{cc}
      0 & 1 \\
      1 & 0\\
    \end{array} \right)$. The set $\overline{X}_4/\ZZ_2$ coincides with the union
of all $\overline{X}_{4,\lambda}$.

The Hodge-Deligne polynomials computed in \cite{lomune} are as follows:
\begin{align*}
 e(X_{0}) &=  q^{4}+4q^{3}-q^2-4q, \\
 e(X_1) & =q^{3}-q ,\\
 e(\overline{X}_2) &= q^3-2q^2-3 q ,\\
 e(\overline{X}_3) & =q^3 +3 q^2 ,\\
 e(\overline{X}_{4,\lambda}) &=q^3 + 3 q^2 - 3 q -1 .\\
\end{align*}

There is a fibration
 $$
 \overline{X}_{4} \rightarrow B= \mathbb{C}-  \lbrace 0,\pm 1 \rbrace.
 $$
Let $\gamma_{0},\gamma_{-1},\gamma_{1}$ 
be loops around the points $0,-1,1$, respectively. The monodromies around 
$\gamma_{1},\gamma_{-1}$ are trivial, and the monodromy
around $\gamma_{0}$ is of order $2$. So the monodromy group is $\Gamma =\mathbb{Z}_{2}$ and the Hodge monodromy 
representation is computed in \cite[Theorem 6.1]{lomune},
 \begin{equation}\label{eqn:RX4}
 R(\overline{X}_{4})=(q^{3}-1)T+(3q^{2}-3q)N,
 \end{equation}
where $T$ is the trivial representation and $N$ the non-trivial one. Then Corollary \ref{cor:general-fibr} gives
 \begin{equation}\label{eqn:e(X4-bar)}
 e(\overline{X}_4) =(q-1)a -2b=  q^4-3q^3-6q^2+5q+3.
 \end{equation}
(where we have written $R(\overline{X}_{4})=aT+bN$, $a=q^3-1$, $b=3q^2-3q$). 

Taking the quotient by $\ZZ_2$, we have the fibration
 $$
  \overline{X}_{4} /\mathbb{Z}_{2} \rightarrow B= 
 (\mathbb{C}-  \lbrace 0, \pm 1 \rbrace )/\mathbb{Z}_{2} \cong \mathbb{C}-  \lbrace \pm 2 \rbrace,
 $$
and the Hodge monodromy representation is a representation of the group $\Gamma=\mathbb{Z}_{2}\times \mathbb{Z}_{2}$, generated by
the loops $\gamma_{-2},\gamma_2$ around the points $-2,2$, respectively. The Hodge monodromy representation is
computed in \cite[Theorem 7.1]{lomune},
 \begin{equation}\label{eqn:RX4Z2}
 R(\overline{X}_{4}/\mathbb{Z}_{2})=q^{3}T-3q S_{2}+3q^{2}S_{-2}-S_{0} ,
 \end{equation}
where $T$ is the trivial representation, $S_{\pm 2}$ is the representation that is non-trivial around the loop $\gamma_{\pm 2}$
and $S_{0}=S_{-2}\otimes S_2$. Now writing $R(\overline{X}_{4}/\mathbb{Z}_{2})=aT+bS_{2}+cS_{-2}+dS_{0}$, 
Corollary \ref{cor:general-fibr} says that
  \begin{equation} \label{formulaepolyx4z2}
  e(\overline{X}_{4}/\mathbb{Z}_{2})=(q-1)a-(a+b+c+d) = q^4-2q^3-3q^2+3q+1 .
  \end{equation}
Note that if $R(\overline{X}_{4}/\mathbb{Z}_{2})=aT+bS_{2}+cS_{-2}+dS_{0}$, then 
$R(\overline{X}_{4})=(a+d)T + (b+c)N$, which agrees with (\ref{eqn:RX4}).

To recover $e(X_4)$, we note that 
 $$
 X_4  \cong (\overline{X}_4 \x \PGL(2,\CC)/D) /\ZZ_2\, .
 $$
By \cite[Proposition 3.2]{lomune}, we have that $ e(\PGL(2, \CC)/D)^+=q^2$ and $ e(\PGL(2, \CC)/D)^-=q$. Using (\ref{eqn:+-}),
 \begin{align}\label{eqn:RX4-RX4Z2->eX4}
 e(X_4) &= q^2 e(\overline{X}_ 4)^+ + q \, e(\overline{X}_4)^- \nonumber \\
  &= q^2 e(\overline{X}_4/\ZZ_2) + q( e(\overline{X}_4)-e(\overline{X}_4/\ZZ_2))  \nonumber \\
 &= (q^2-q)e(\overline{X}_4/\ZZ_2) +q \, e(\overline{X}_{4}).
 \end{align}
Now using (\ref{formulaepolyx4z2}) and (\ref{eqn:e(X4-bar)}) we get
 $$
  e(X_4)  = q^6-2q^5-4q^4+3q^2+2q\, .
 $$

\section{E-polynomial of the twisted character variety}\label{sec:twisted}

We start by computing the E-polynomial of the twisted character variety $\mathcal{M}^1=\cM^1(\SL(2,\CC))$ for a curve $X$ of genus $3$. 
This space can be described as the quotient
 \begin{equation}\label{eqn:twisted}
 \mathcal{M}^1=W/\pgl,
 \end{equation}
where 
 $$
 W=\lbrace (A_1,B_1,A_2,B_2,A_3,B_3)\in \SL(2,\CC)^6 \mid [A_1,B_1][A_2,B_2][A_3,B_3]=-\text{Id} \rbrace.
 $$
Notice that we only need to consider a geometric quotient, since all representations are irreducible: if there was 
a common eigenvector $v$ for $(A_1,B_1,A_2,B_2,A_3,B_3)$, we would obtain that $[A_1,B_1](v)=[A_2,B_2](v)=[A_3,B_3](v)=v$ and therefore
$[A_1,B_1][A_2,B_2][A_3,B_3](v)=v\neq -\Id(v)=-v$.

To compute the E-polynomial of the moduli space, we will stratify the space in locally closed subvarieties and 
compute the E-polynomial of each stratum, since E-polynomials are additive with respect to stratifications. 
We stratify the space $W$ according to the possible values of the traces of the commutators. To simplify 
the notation, we write $[A_1,B_1]=\xi_{1}, [A_2,B_2]=\xi_{2}, [A_3,B_3]=\xi_{3}$. Consider the map
 \begin{eqnarray*}
 F: W & \longrightarrow & \mathbb{C}^{3} \\
 (A_1,B_1,A_2,B_2,A_3,B_3) & \mapsto & (t_{1},t_{2},t_{3})=(\tr \xi_1,\tr \xi_2,\tr \xi_3).
 \end{eqnarray*}

We are interested in the following condition: if $\xi_{1},\xi_{2},\xi_{3}$ share 
an eigenvector $v$ with eigenvalue $\lambda_{i}$ in each case, then in a suitable basis
 $$
 [A_1,B_1]=\begin{pmatrix} \lambda_{1} & \ast \\ 0 & \lambda_{1}^{-1} \end{pmatrix}, \quad 
 [A_2,B_2]=\begin{pmatrix} \lambda_{2} & \ast \\ 0 & \lambda_{2}^{-1} \end{pmatrix}, \quad
 [A_3,B_3]= \begin{pmatrix} \lambda_{3} & \ast \\ 0 & \lambda_{3}^{-1} \end{pmatrix},
 $$
and therefore, as $\xi_1\xi_2\xi_3=-\Id$, we obtain that $\lambda_{1}\lambda_{2}\lambda_{3}=-1$. Other
possibility is that $v$ has eigenvalue $\lambda_1$ for $\xi_1$, $\lambda_2$ for $\xi_2$ and
$\lambda_3^{-1}$ for $\xi_3$, yielding $\lambda_1\lambda_2\lambda_3^{-1}=-1$, etc.
Working out all possibilities, we have that $\xi_1,\xi_2,\xi_3$ can share an eigenvector when
 \begin{equation} \label{eqntracesbis}
 \lambda_{3}=-\lambda_{1}\lambda_{2},\: 
 \lambda_{3}= -\lambda_{1}^{-1}\lambda_{2},\:  
 \lambda_{3}=-\lambda_{1} \lambda_{2}^{-1},
 \text{ or } \lambda_3=-\lambda_{1}^{-1}\lambda_{2}^{-1}
 \end{equation}
Equivalently, in terms of the traces, when
 \begin{equation} \label{eqntraces}
 t_{1}^{2}+t_{2}^{2}+t_{3}^{2}+t_{1}t_{2}t_{3}=4. 
 \end{equation}
This defines a (smooth) quadric $C\subset \mathbb{C}^{3}$, depicted in Figure \ref{fig1}.
 \begin{center}
 \begin{figure}[h]
 \includegraphics[width=7 cm]{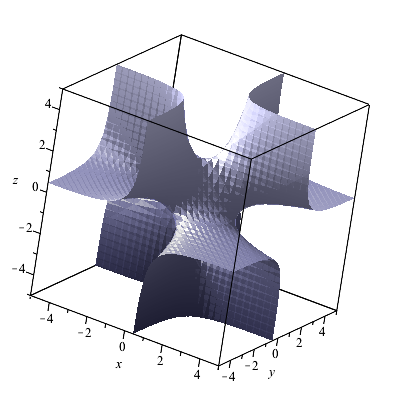}
 \caption{The quadric $C$.} \label{fig1}
 \end{figure}
\end{center}

\begin{lem} \label{lem:C}
$\xi_{1},\xi_{2},\xi_{3}$ share an eigenvector if and only if $(t_{1},t_{2},t_{3})\in C$.
\end{lem}

\begin{proof}
We have already seen the if part. For the only if part, note that if $(t_1,t_2,t_3)$ satisfies 
(\ref{eqntraces}) then the eigenvalues $\lambda_1,\lambda_2,\lambda_3$ satisfy one
of the four relations in (\ref{eqntracesbis}). Changing some $\lambda_i$ by $\lambda_i^{-1}$ if
necessary, we can suppose that $\lambda_3=-\lambda_1\lambda_2$. 

Suppose that some $\xi_i$ is diagonalizable. Without loss of generality, we suppose it is $\xi_1$,
and choose a basis with $\xi_1= \begin{pmatrix} \lambda_1 & 0 \\ 0 & \lambda_1^{-1} \end{pmatrix} $. 
Write $\xi_2=\begin{pmatrix} a & b \\ c & d \end{pmatrix}$.  Then we have the equation
 $$
\xi_3^{-1}=-\xi_1\xi_2= - \begin{pmatrix} \lambda_1 & 0 \\ 0 & \lambda_1^{-1} \end{pmatrix} 
 \begin{pmatrix} a & b \\ c & d   \end{pmatrix}= 
 - \begin{pmatrix} \lambda_1 a & \lambda_1 b \\  
 \lambda_1^{-1}c & \lambda_1^{-1} d   \end{pmatrix} .
 $$
We get the equations $t_{2}=a+d$ and $t_{3}=\tr \xi_3=\tr \xi_3^{-1}=-\lambda_1 a -\lambda^{-1}_1 d$. 
Hence $\lambda_2+\lambda_2^{-1}=a+d$ and $\lambda_1\lambda_2 +\lambda_1^{-1}
\lambda_2^{-1}=\lambda_1 a +\lambda^{-1}_1 d$. If $\lambda_1\neq \pm 1$, it must
be $a=\lambda_2,d=\lambda_2^{-1}$; thus $ad=1$
and hence $bc=0$, which implies that the matrices share an eigenvector.
If $\lambda_1=\pm 1$, then $\xi_1=\pm \Id$, and $\xi_2=\pm \xi_3$, so the matrices
share their eigenvectors.

Now suppose that all $\xi_i$ are not diagonalizable, so they are of Jordan type. Let $v_i$ be
the only eigenvector (up to scalar multiples) of $\xi_i$. If $\xi_i$ do not share an eigenvector,
then $v_1,v_2$ is a basis, on which $\xi_1=\lambda_1\begin{pmatrix} 1 & b \\ 0 & 1 \end{pmatrix}$,
$\xi_2=\lambda_2 \begin{pmatrix} 1 &  0 \\ c & 1 \end{pmatrix}$, where $\lambda_i\in \{\pm 1\}$.
Then $\xi_3^{-1}=-\lambda_1\lambda_2 \begin{pmatrix} 1+bc & b \\ c & 1 \end{pmatrix}$ hence
$t_3=2\lambda_3=-2\lambda_1\lambda_2 = -\lambda_1\lambda_2 (2+bc)$. So $bc=0$, which is
a contradiction.
\end{proof}

We stratify the space $W$ according to the traces of $\xi_1,\xi_2,\xi_3$. Consider the planes
$t_i=\pm 2$ and the quadric $C$ above. Then consider as well the intersections of these seven
subvarieties. This gives the required stratification. We shall compute the E-polynomials of the
chunk of $W$ lying above each of these strata, starting by the lower-dimensional ones
(points) and going up in dimension.

We shall use some of the polynomials computed in Sections 8-12 in \cite{lomune} which
correspond to some of the strata of the spaces of representations for the case of a curve 
of genus $g=2$. We shall point out which stratum we use each time.

\subsection{Special points} \label{subsec:points}

The intersections of three planes, or of the quadric and two planes, is the collection of eight points given by
$(t_{1},t_{2},t_{3})=(\pm 2, \pm 2, \pm 2)$.  Let $W_{1}$ be the subset of $(A_1,B_1,A_2,B_2,A_3,B_3)\in W$ with traces
given by these possibilities. Note that if $t_i=2$ then $\xi_i=\Id$ or $\xi_i$ is of Jordan type $J_+$; analogously,
 if $t_i=-2$ then $\xi_i=-\Id$ or $\xi_i$ is of Jordan type $J_-$.

 \begin{itemize}
 \item $W_{11}=\lbrace (t_{1},t_{2},t_{3})=(2,2,2) \rbrace = F^{-1}((2,2,2))$. Then $\xi_{i}$ are all of Jordan type 
(if $\xi_1=\Id$ then  $\xi_{2}\xi_{3}=-\Id$, so $\xi_{2}=-\xi_{3}^{-1}$ and $t_{2}=-t_{3}$).  
Choosing an adequate basis, we can assume that $[A_1,B_1]=J_{+}$, so $[A_2,B_2][A_3,B_3]=-(J_{+})^{-1} = J_{-}$. 
This set has been determined in \cite[Stratum $Z_3$, Section 12]{lomune}. It has polynomial $q\, e(\overline{X}_2)^2$. Therefore 
 \begin{align*}
 e(W_{11}) & = q\, e(\overline{X}_{2})^3 e(\pgl/U) \\ 
  & = q^{12}-6q^{11}+2q^{10} +34q^9-12q^8-82q^7-18q^6+54q^5+27q^4.
 \end{align*}

\item $W_{12}= \lbrace (t_{1},t_{2},t_{3})=(2,2,-2) \rbrace = F^{-1}((2,2,-2))$ and the cyclic permutations 
(accounting for three cases). If $\xi_{1}=\Id$, then $[A_2,B_2][A_3,B_3]=-\Id$, with $t_{2}=2$, and $t_{3}=-2$. If $\xi_2 =\Id$ then $\xi_3=-\Id$;
and if $\xi_2\sim J_+$ then $\xi_3\sim J_-$. This produces the contribution (multiplying by $3$ because of the three cyclic permutations)
  \begin{align*}
   e(W_{12}') & =  3e({X}_{0})(e(X_{0})e(X_{1})+e(\pgl/U)e(\overline{X}_{2})e(\overline{X}_{3})) \\
  & = 3q^{12}+18q^{11}+3q^{10}-126q^9 -147q^8+150q^7+273q^6+6q^5-132q^4-48q^3. 
  \end{align*}
If $\xi_{1}\sim J_{+}$, then 
choosing an adequate basis, we can assume that $[A_1,B_1]=J_{+}$, so $[A_2,B_2][A_3,B_3]=-(J_{+})^{-1} = J_{-}$. 
This set has been determined in \cite[Stratum $Z_1$, Section 12]{lomune} to be
$(q-2)e(\overline{X}_{2})e(\overline{X}_{3})+e(\overline{X}_{2})e({X}_{1})+e(\overline{X}_{3})
e({X}_{0})$. Therefore
 \begin{align*}
 e(W_{12}'') & = 3e(\overline{X}_{2})\left( (q-2)e(\overline{X}_{2})e(\overline{X}_{3}) 
 +e(\overline{X}_{2})e({X}_{1})+e(\overline{X}_{3})e({X}_{0}) \right) e(\pgl/U) \\
 & = 6q^{12}+9q^{11}-72q^{10}-66q^9 +120q^8+36q^7-144q^6-6q^5+90q^4+27q^3.
\end{align*}
Finally
\begin{align*}
 e(W_{12}) & = e(W_{12}')+e(W_{12}'') \\ 
 & = 9q^{12}+27q^{11}-69q^{10}-192q^9 -27q^8+186q^7+129q^6-42q^4-21q^3.
\end{align*}

\item $W_{13}=\lbrace (t_{1},t_{2},t_{3})=(-2,-2,2) \rbrace = F^{-1}((-2,-2,2))$ and cyclic permutations 
(which account for three cases). If $\xi_{1}=-\Id$ then it must be $t_{2}=t_{3}$, which is not the case. 
Therefore $\xi_{1}$ is of Jordan type $J_{-}$. Choosing a suitable basis, we can write $\xi_{1}=J_{-}$, 
so $[A_1,B_1][A_2,B_2]=-J_{-}^{-1}= J_+$. 
This set has been determined in \cite[Stratum $Z_3$, Section 11]{lomune} and it has polynomial 
$q\, e(\overline{X}_2) e(\overline{X}_3)$. Therefore 
 \begin{align*}
 e(W_{13}) & =  3 q\, e(\overline{X}_{3})^2 e(\overline{X}_{2})e(\pgl /U ) \\ 
  &= 3q^{12}+12q^{11}-21q^{10}-120q^9-63q^8+108q^7+81q^6.
\end{align*}

\item $W_{14}=\lbrace (t_{1},t_{2},t_{3})=(-2,-2,-2) \rbrace=F^{-1}((-2,-2,-2))$. 
If $\xi_{1}=-\Id$, then $[A_2,B_2][A_3,B_3]=\Id$, with $t_{2}=-2$, and $t_{3}=-2$. If $\xi_2 =-\Id$ then $\xi_3=-\Id$;
and if $\xi_2\sim J_-$ then $\xi_3\sim J_-$. This produces the contribution 
  \begin{align*}
   e(W_{14}') & =  e({X}_{1})\left( e(X_{1})^2+e(\pgl/U)e(\overline{X}_{3})^2 \right) \\
  & = q^{11}+6q^{10}+8q^9-12q^8-20q^7+6q^6+12q^5-q^3.
\end{align*}
If $\xi_{1}\sim J_{-}$, then choosing an adequate basis, we can assume that $[A_1,B_1]=J_{-}$, so $[A_2,B_2][A_3,B_3]=-(J_{-})^{-1} = J_{+}$. 
This set has been determined in \cite[Stratum $Z_2$, Section 11]{lomune} to be
$(q-2)e(\overline{X}_{3})^2+2e(\overline{X}_{3})e({X}_{1})$. Therefore
 \begin{align*}
  e(W_{14}'')  & = e(\overline{X}_{3}) \left( (q-2)e(\overline{X}_{3})^{2}+2e(\overline{X}_{3})e({X}_{1}) \right)e(\pgl/U) \\
  & = q^{12}+9q^{11}+20q^{10}-20q^9-87q^8-7q^7+66q^6+18q^5.
\end{align*}
Hence
 \begin{align*}
 e(W_{14}) = e(W_{14}')+e(W_{14}'')= q^{12}+10q^{11}+26q^{10}-12q^9-99q^8-27q^7+72q^6+30q^5-q^3.
 \end{align*}
\end{itemize}

Adding all up, we obtain 
 \begin{align*}
 e(W_{1})  = 14q^{12}+43q^{11}-62q^{10}-290q^{9}  -201q^{8}+185q^{7}+264q^{6}+84q^{5}-15q^{4}-22q^{3}.
 \end{align*}

\subsection{Special lines} \label{subsec:lines}

The intersection of two of the planes are the lines 
$\lbrace (t_{1},t_{2},t_{3})=(\pm 2, \pm 2,t_{3}), t_{3}\in \CC-\{ \pm 2\} \rbrace$, and the cyclic permutations of these.
The intersection of one of the planes and the quadric is given by the lines
$\lbrace t_{1}=\pm 2$, $t_{2}=\mp t_{3}\in \CC-\{ \pm 2\} \rbrace$ and the cyclic permutations. We denote by
$W_2$ the portion of $W$ lying over these lines, and we stratify in the following sets:

\begin{itemize}
\item $W_{21}$ given by $t_1= 2$, $t_2=\pm 2$ and $t_3\in \CC-\{\pm2\}$.
Note that if $\xi_1= \Id$ then $\xi_{2}\xi_{3}=- \Id$, with $\tr\xi_{2}=\pm 2$ and $\tr \xi_{3}\neq \pm2$, a contradiction. Then $\xi_1$ is of Jordan type. 
Choosing an adequate basis, we can assume that $\xi_{1}=J_{+}$. Then $[A_2,B_2][A_3,B_3]=-(J_+)^{-1}=J_{-}$, 
where $\tr \xi_2=\pm 2$ and $\tr \xi_3\neq \pm 2$. This set has been computed in \cite[Stratum $Z_{4}$, Section 12]{lomune}
and it has E-polynomial $q\, (e(\overline{X}_{2})+e(\overline{X}_{3}))e(\overline{X}_{4}/\mathbb{Z}_{2})$. So 
\begin{align*}
e(W_{21}) & = e(\overline{X}_{2})e(\pgl /U)q\, (e(\overline{X}_{2})+e(\overline{X}_{3}))e(\overline{X}_{4}/\mathbb{Z}_{2})\\ 
&=  2q^{13}-7q^{12}-13q^{11}+47q^{10}+40q^9 -103q^8-58q^7+93q^6+38q^5-30q^4-9q^3.
\end{align*}

\item $W_{22}$ given by $t_1= -2$, $t_2=\pm 2$ and $t_3\in \CC-\{\pm2\}$. Again $\xi_1$ is of Jordan type. 
Fixing a basis we have $\xi_1=J_-$ and $\xi_{2}\xi_{3}=-J_-^{-1}= J_{+}$, 
where $\tr \xi_{2}=\pm 2$ and $\tr \xi_{3} \neq \pm 2$. 
This set has been computed in \cite[Stratum $Z_{4}$, Section 11]{lomune}
and it has E-polynomial $q\, (e(\overline{X}_{2})+e(\overline{X}_{3}))e(\overline{X}_{4}/\mathbb{Z}_{2})$. So 
\begin{align*}
 e(W_{22}) & =  e(\overline{X}_{3})e(\pgl /U) q(e(\overline{X}_{2})+e(\overline{X}_{3}))e(\overline{X}_{4}/\mathbb{Z}_{2}) \\
 & = 2q^{13}+3q^{12}-22q^{11}-27q^{10}+61q^9 +58q^8-68q^7-43q^6+27q^5+9q^4.
\end{align*}

\item $W_{23}$ given by $t_1=2$, $\xi_{1}=\Id$ and $t_{2}=-t_{3}\in \CC-\{ \pm 2\}$. Now
$\xi_{2}\xi_{3}=-\Id$, $t_{2}=-t_{3}\neq \pm 2$. This is computed in \cite[Stratum $W_4$, Section 9]{lomune},
$e(W_4)=q^{9}-2q^{8}-7q^{7}  -18q^{6}+24q^{5}+28q^{4}-17q^{3}-8q^{2}-q$. So 
 \begin{align*}
 e(W_{23}) & = e(X_{0})e(W_{4}) \\ 
 & = q^{13}+2q^{12}-16q^{11}-48q^{10}-33q^9+170q^8  +143q^7-200q^6-128q^5+72q^4+33q^3+4q^2.
 \end{align*}

\item $W_{24}$ given by $t_1=2$, $\xi_{1}\sim J_{+}$. Using an adequate basis, we have
$\xi_1=J_+$, and $\xi_{2}\xi_{3}=-J_+^{-1} = J_{-}$, and $t_{2}=-t_{3}\neq \pm 2$. 
This is computed in \cite[Stratum $Z_5$, Section 12]{lomune},
$e(Z_5)=q^{8}-3q^{7} -3q^{6}-35q^{5}+69q^{4}-15q^{3}-11q^{2}-3q$. So 
\begin{align*}
 e(W_{24}) & = e(\overline{X}_{2})e(\pgl /U) e(Z_{5}) \\ 
 & = q^{13}-5q^{12}-q^{11}-15q^{10}+148q^9-28q^8-336q^7+112q^6+227q^5-55q^4-39q^3-9q^2. 
\end{align*}

\item $W_{25}$ given by $t_1=-2$, $\xi_{1}=-\Id$. Then $\xi_{2}\xi_{3}=\Id$, with $t_{2}=t_{3}\neq \pm 2$. 
This is computed in \cite[Stratum $Y_4$, Section 8]{lomune},\footnote{We make a correction to $e(Y_{4})$ in \cite{lomune}, $e(Y_{4})=q^{9}-2q^{8}+2q^{7}-18q^{6}+6q^{5}+28q^{4}-8q^{3}-8q^{2}-q$.}
$e(Y_4)=q^{9}-2q^{8}+2q^{7}-18q^{6}+6q^{5}+28q^{4}-8q^{3}-8q^{2}-q$. So
\begin{align*}
 e(W_{25}) & =  e({X}_{1})e(Y_{4}) \\ 
 & = q^{12}-2q^{11}+q^{10}-16q^9+4q^8+46q^7-14q^6-36q^5+7q^4+8q^3+q^2.
\end{align*}

\item $W_{26}$ given by $t_1=-2$, $\xi_{1} \sim J_{-}$. Choosing a basis so that $\xi_1=J_-$, 
$\xi_{2}\xi_{3}=J_{+}$, with $t_{2}=t_{3} \neq \pm 2$. 
This is computed in \cite[Stratum $Z_5$, Section 11]{lomune},
$e(Z_5)=q^{8}-3q^{7}-3q^{6}-35q^{5}+69q^{4}-15q^{3}-11q^{2}-3q$. So
 \begin{align*}
e(W_{26}) & = e(\overline{X}_{3})e(\pgl /U) e(Z_{5}) \\ 
 & = q^{13}-13q^{11}-44q^{10}-24q^9 +236q^8-20q^7-228q^6+47q^5+36q^4+9q^3.  
\end{align*}
\end{itemize}

Considering the possible permutations, and adding the E-polynomials of the strata, we get
\begin{align*}
  e(W_{2}) & =  3e(W_{21})+3e(W_{22})+3e(W_{23})+3e(W_{24})+3e(W_{25})+3e(W_{26}) \\ 
  & = 21q^{13}-18q^{12}-201q^{11}-258q^{10}+528q^9+1011q^8-879q^7-840q^6+525q^5+117q^4+6q^3-12q^2.
\end{align*}

\subsection{Special planes} \label{subsec:planes}

Now consider the planes given by the equations $\lbrace t_{i}=\pm 2 \rbrace$, 
from which we remove the previous strata. We do the case $t_1=\pm 2$ and multiply by three the
result to account for the three cases $i=1,2,3$. The planes are given by 
$t_1=\pm 2$, where $t_{2},t_{3} \neq \pm 2$ and $t_{2} \neq \mp  t_{3}$. 
Note that it cannot be $\xi_{1}=\pm \Id$, since this would imply $t_{2}=\mp t_{3}$. 
We have the following cases:

\begin{itemize}
\item $W_{31}$ given by $\xi_{1}\sim J_{+}$. This implies that, in a suitable basis, $\xi_{2}\xi_{3}=J_{-}$, 
together with the previous restrictions for the traces. This is the set given
in \cite[Stratum $Z_{6}$, Section 12]{lomune}, which has E-polynomial
$e(Z_6)=q^{9}-5q^{8}+24q^{6} +20q^{5}-60q^{4}+6q^{3}+11q^{2}+3q$. So
 \begin{align*}
 e(W_{31}) & = e(\overline{X}_{2})e(\pgl /U) e(Z_{6}) \\  
 & = q^{14}-7q^{13}+6q^{12}+46q^{11}-35q^{10}-211q^9+94q^8+351q^7-103q^6-218q^5+28q^4+39q^3+9q^2.
 \end{align*}

\item $W_{32}$ given by $\xi_{1} \sim J_{-}$. This implies that, in a suitable basis, $\xi_{2} \xi_{3}=J_{+}$, 
together with the restrictions $t_{2}, t_{3} \neq \pm 2$, $t_{2}\neq t_{3}$.
This is the set given in \cite[Stratum $Z_{6}$, Section 11]{lomune}, which has E-polynomial
$e(Z_6)=q^{9}-5q^{8}+15q^{6} +11q^{5}-51q^{4}+15q^{3}+11q^{2}+3q$. So
 \begin{align*}
 e(W_{32}) & = e(\overline{X}_{3})e(\pgl /U) e(Z_{6}) \\ 
 & =  q^{14}-2q^{13}-16q^{12}+17q^{11}+71q^{10}-33q^9-194q^8+74q^7+174q^6-47q^5-36q^4-9q^3.
 \end{align*}
\end{itemize}
The total contribution of the planes is
\begin{align*}
 e(W_{3}) &= 3(e(W_{31})+e(W_{32}))  \\ &= 6q^{14}-27q^{13}-30q^{12}+189q^{11}+108q^{10}-732q^9
 -300q^8+1275q^7+213q^6-795q^5-24q^4+90q^3+27q^2.
\end{align*}

\subsection{The quadric $C$. One eigenvector} \label{subsec:quadric}

We now deal with the part of $W$ lying over
the quadric $C$, with $t_{1},t_{2},t_{3} \neq \pm 2$. As we saw in Lemma 
\ref{lem:C}, this corresponds to the case that $\xi_{1},\xi_{2}, \xi_{3}$ share 
(at least) one eigenvector.

We deal now with the case where $\xi_{1},\xi_{2}$ and $\xi_{3}$ share just one eigenvector (up to scalar multiples), which we denote by $v$. 
Using it as the first vector of a basis that diagonalizes $\xi_{1}$, we can arrange that
 \begin{equation}
 \xi_{1}= \begin{pmatrix} \lambda & 0 \\ 0 & \lambda^{-1} \end{pmatrix}, \: 
 \xi_{2}= \begin{pmatrix} \mu & 1 \\ 0 & \mu^{-1}\end{pmatrix}, \: 
 \xi_{3}=\begin{pmatrix} -\lambda^{-1} \mu^{-1} & \lambda \\ 0 &  -\lambda \mu \end{pmatrix} \label{eqfirstcase} 
 \end{equation}
where $\lambda=\lambda_1$, $\mu=\lambda_2$ are the eigenvalues of $\xi_{1},\xi_{2}$ associated to the eigenvector $v$. 
Our choice of basis for the above expressions of $\xi_{1},\xi_{2},\xi_{3}$ gives us a slice 
for the $\pgl$-action, since their stabilizer is trivial. So
we shall only need to multiply by $e(\pgl)$ after computing the E-polynomial of \eqref{eqfirstcase}. 
This set can be regarded as a fibration 
  $$
  Z \too B= \lbrace (\lambda,\mu)\in (\mathbb{C}^{\ast})^{2} \mid \lambda,\mu,\lambda\mu \neq \pm 1 \rbrace,
 $$ 
with fiber 
 $$
 \overline{X}_{4,\lambda} \times \overline{X}_{4,\mu} \times \overline{X}_{4,-\lambda^{-1}\mu^{-1}}.
 $$
To compute its E-polynomial, we would like to extend the fibration to the six curves $\lambda,\mu,\lambda\mu=\pm 1$ and 
apply Corollary \ref{cor:thm1}. We cannot extend the fibration; however we can extend the local system definining the Hodge
monodromy fibration. By (\ref{eqn:RX4}), the Hodge monodromy fibration $R(\overline{X}_4)$ is trivial over $\lambda=\pm 1$, and it
is of order $2$ over $\lambda=0$. Consider the projections:
  \begin{eqnarray*}
 \pi_{1}: B & \longrightarrow &  \mathbb{C}^{\ast} - \lbrace \pm 1 \rbrace \\
(\lambda, \mu) & \mapsto & \lambda \\
\pi_{2}: B & \longrightarrow &  \mathbb{C}^{\ast} -\lbrace \pm 1 \rbrace \\
(\lambda, \mu) & \mapsto & \mu \\
\pi_{3}: B & \longrightarrow &  \mathbb{C}^{\ast} - \lbrace \pm 1 \rbrace \\
(\lambda, \mu) & \mapsto & - \lambda^{-1} \mu^{-1}
  \end{eqnarray*}
Then 
  $$
 \overline Z= \pi_{1}^{\ast}(\overline{X}_{4})\times \pi_{2}^{\ast}(\overline{X}_{4}) \times \pi_{3}^{\ast}(\overline{X}_{4}).
  $$
The Hodge monodromy fibration $R(\overline Z)$ can be extended (locally trivially) over the lines $\lambda =\pm 1$, 
$\mu=\pm 1$ and $\lambda\mu=\pm 1$, to a 
Hodge monodromy fibration $\tilde R(\overline Z)$ over $\tilde B=\CC^*\x \CC^*$. Moreover, the monodromy around 
$\lambda=0$ and $\mu=0$ is of order two.
The corresponding group is $\Gamma=\ZZ_2\x \ZZ_2$ and the representation ring is generated by 
representations $T, N_{1},N_{2}, N_{12}=N_{1}\otimes N_{2}$,
where $T$ is the trivial representation, $N_1$ is the representation with non-trivial monodromy around the origin of the first copy of $\CC^*$, and 
$N_2$ is the representation with non-trivial monodromy around the origin of the second copy of $\CC^*$.

Pulling back the Hodge monodromy representation of ${\overline{X}}_{4}$ given in (\ref{eqn:RX4}), 
we have that $R(\pi_1^*(\overline{X}_4))=aT+bN_{1}$,
$R(\pi_2^*(\overline{X}_4))=aT+bN_{2}$ and $R(\pi_3^*(\overline{X}_4))=aT+bN_{12}$, where $a=q^3-1$, $b=3q^2-3q$.
So the Hodge monodromy representation of $\overline Z$ is
  \begin{align*}
  R(\overline Z) & = (aT+bN_{1})\otimes (aT+bN_{2}) \otimes (aT+bN_{12})  \\
  & =  (a^{3}+b^{3})T+(a^{2}b+ab^{2})N_{1}+(a^{2}b+ab^{2})N_{2}+(a^{2}b+ab^{2})N_{12} .
\end{align*}

We extend $R(\overline Z)$ to a Hodge monodromy fibration $\tilde R(\overline Z)$ over $\tilde B=\CC^*\x \CC^*$
with the same formula, and 
compute its E-polynomial, applying Corollary \ref{cor:thm1},
\begin{align*}
e(\tilde{R}(\overline Z)) & =  (q-1)^{2}e(F)^{inv} =  (q-1)^{2}(a^{3}+b^{3}) 
 =  (q-1)^{2}((q^3-1)^{3}+(3q^{2}-3q)^{3}) \\
& =  q^{11}-2q^{10}+q^9+24q^8-129q^7+267q^6-267q^5+129q^4-24q^3-q^2+2q-1.
\end{align*}

Now we substract the contribution of $\tilde R(\overline Z)$ over the lines $\lambda=\pm 1$, $\mu=\pm 1$ and $\lambda\mu=\pm 1$.

\begin{itemize}
\item Consider the curve defined by $\lambda=1$, $\mu\neq \pm 1$. The fibration has the Hodge monodromy of 
a fibration over $\CC^*-\{\pm 1\}$ with fiber
 $$
 \overline{X}_{4,\lambda_{0}}\times \overline{X}_{4,\mu} \times \overline{X}_{4,-\mu^{-1}} 
 $$
This has Hodge monodromy representation equal to $e(\overline{X}_{4,\lambda_0})
R(\overline{X}_{4})\otimes \tau^{\ast}R(\overline{X}_{4})$, where $\tau(\mu)=-\mu^{-1}$. This is equal to
\begin{align*}
 e(\overline{X}_{4,\lambda_0})R(\overline{X}_{4})\otimes \tau^{\ast}R(\overline{X}_{4}) 
 & =  e(\overline{X}_{4,\lambda_0})R(\overline{X}_{4})\otimes R(\overline{X}_{4}) \\
 & =   e(\overline{X}_{4,\lambda_0}) (aT+bN)\otimes (aT+bN) \\
 & =  e(\overline{X}_{4,\lambda_0}) ((a^{2}+b^{2})T + (2ab)N). 
 \end{align*}
Using Corollary \ref{cor:general-fibr} and (\ref{eqn:e(X4-bar)}), we get that the contribution equals
 \begin{align} \label{eqn:333}
 e(\overline{X}_{4,\lambda_0}) &\left((q-3)(a^{2}+b^{2}) -2 (2ab) \right) \\ 
&= q^{10}-15q^8-36q^7-24q^6+300q^5 -238q^4-60q^3+39q^2+20q+3. \nonumber
 \end{align}
 
\item The computation for $\lambda=-1$, $\mu \neq \pm1$ is analogous and gives the same quantity. 

\item By symmetry, the contribution for $\mu=\pm 1$, and for $\lambda\mu=\pm1$ is the same as for $\lambda=\pm 1$.
So we have to multiply (\ref{eqn:333}) by $6$.

\item The contribution of the four points $(\pm 1, \pm 1)$ is
$4e(X_{4,\lambda_{0}})^{3} = 4q^9+36q^8+72q^7-120q^6-288q^5+288q^4+120q^3-72q^2-36q-4$.
\end{itemize}

Therefore, the E-polynomial of the original fibration is:
 \begin{align*}
 e(\overline Z) & =  e(\tilde{R}(\overline Z)) - 6 (q^{10}-15q^8-36q^7-24q^6+300q^5 -238q^4-60q^3+39q^2+20q+3)-4e(X_{4,\lambda_{0}})^{3} \\
   & = q^{11}-8q^{10}-3q^9+78q^8+15q^7+531q^6 -1779q^5+1209q^4+216q^3-163q^2-82q-15
 \end{align*}
and
\begin{align*}
e(W_{4}) & = e(\pgl)e(\overline Z) \\
 &= q^{14}-8q^{13}-4q^{12}+86q^{11}+18q^{10}+453q^9-1794q^8 
+678q^7+1995q^6-1372q^5-298q^4+148q^3+82q^2+15q.
\end{align*}

\subsection{The quadric $C$. Two eigenvectors} \label{subsec:quadric-two}

Suppose now that $(t_1,t_2,t_3)\in C$ and $\xi_{1},\xi_{2},\xi_{3}$ share two eigenvectors. Then they can all 
be simultaneously diagonalized. With respect to a suitable basis, we have that:
 $$
 [A_1,B_1]= \begin{pmatrix} \lambda & 0 \\ 0 & \lambda^{-1} \end{pmatrix}, \quad 
 [A_2,B_2]=\begin{pmatrix} \mu & 0 \\ 0 & \mu^{-1} \end{pmatrix}, \quad 
 [A_3,B_3]=\begin{pmatrix} -\lambda^{-1} \mu^{-1} & 0 \\ 0 & -\lambda \mu \end{pmatrix}.
 $$
This defines a fibration $\overline{Z} \to {B}:= 
\lbrace (\lambda, \mu) \in (\mathbb{C}^{\ast})^2 |  \lambda,\mu, \lambda \mu \neq \pm 1 \rbrace$, 
with fiber
 $$
 \overline{X}_{4,\lambda} \times \overline{X}_{4,\mu} \times \overline{X}_{4,-\lambda^{-1}\mu^{-1}} 
 $$
The stabilizer of $\xi_{1},\xi_{2},\xi_{3}$ is $D\x \ZZ_2$, where $D$ are the diagonal matrices
and  the $\mathbb{Z}_{2}$-action is 
given by the simultaneous permutation of the eigenvalues, i.e, by conjugation by $P_{0}=\begin{pmatrix} 0 &1 \\ 1 &0\end{pmatrix}$. 
Therefore the stratum we are analysing is 
  $$
  W_5\cong Z=( \overline{Z} \times \pgl/D)/\ZZ_2\, .
 $$ 
The action on the basis of $\overline{Z} \to B$ takes $(\lambda,\mu)$ to $(\lambda^{-1},\mu^{-1})$, producing a 
fibration
  \begin{equation}\label{eqn:barZ}
 \xymatrix{ \overline{Z} \ar[r] \ar[d] & \overline{Z}':=\overline{Z}/\mathbb{Z}_{2} \ar[d]  \\ 
 {B} \ar[r] & {B}'= {B}/\mathbb{Z}_{2}}
  \end{equation}
If we write
 \begin{eqnarray}
 \pi_{1}: B' & \longrightarrow &  \mathbb{C}^{\ast}/ \mathbb{Z}_{2} \nonumber \\
 (\lambda, \mu) & \mapsto & \lambda \nonumber \\
 \pi_{2}: B' & \longrightarrow &  \mathbb{C}^{\ast}/ \mathbb{Z}_{2} \label{eqn:projections} \ \\
 (\lambda, \mu) & \mapsto & \mu \nonumber \\
 \pi_{3}: B' & \longrightarrow &  \mathbb{C}^{\ast}/ \mathbb{Z}_{2} \nonumber \\
 (\lambda, \mu) & \mapsto & \lambda^{-1} \mu^{-1} \nonumber
 \end{eqnarray}
($\ZZ_2$ acts on $\CC^*$ by $x\sim x^{-1}$), we can obtain three 
pullback bundles 
  \begin{equation*}
 \xymatrix{ \overline{Z}_{i}'\ar[r] \ar[d] &  \overline{X}_{4}/\mathbb{Z}_{2} \ar[d]  \\ 
 {B}' \ar[r]^>>>>>{\pi_{i}} & (\mathbb{C}^* -  \lbrace \pm 1 \rbrace) /\mathbb{Z}_{2} }
   \end{equation*}
$i=1,2,3$, such that $\overline{Z}'\cong \overline{Z}_{1}'\times \overline{Z}_{2}'\times \overline{Z}_{3}' 
\cong \pi_{1}^{\ast}(\overline{X}_{4}/\mathbb{Z}_{2})\times \pi_{2}^{\ast}(\overline{X}_{4}/\mathbb{Z}_{2}) \times 
f^*\pi_{3}^{\ast}(\overline{X}_{4}/\mathbb{Z}_{2})$, with $f(x)=-x$.

As a consequence, if we write $R(\overline{X}_{4}/\mathbb{Z}_{2})=aT+bS_{2}+cS_{-2}+dS_{0}$, as
in (\ref{eqn:RX4Z2}), the Hodge monodromy representation is
  \begin{align} 
  R(\overline{Z}') =& \pi_{1}^{\ast}(R(\overline{X}_{4}/\mathbb{Z}_{2}))\otimes \pi_{2}^{\ast}(R(\overline{X}_{4}/\mathbb{Z}_{2})) 
  \otimes f^*\pi_{3}^{\ast}(R(\overline{X}_{4}/\mathbb{Z}_{2})) \label{eqnlongrepresentations} \\
  =& (aT+bS_{2}^{\lambda}+cS_{-2}^{\lambda}+dS_{0}^{\lambda}) \otimes (aT+bS_{2}^{\mu}+cS_{-2}^{\mu}+dS_{0}^{\mu}) \otimes
  (aT+cS_{2}^{\lambda\mu}+bS_{-2}^{\lambda\mu}+dS_{0}^{\lambda\mu}) \nonumber \\
 = &  a^{3}T+a^{2}bS_{2}^{\mu}+a^{2}cS_{-2}^{\mu} + a^{2}dS_{0}^{\mu} 
 +a^{2}bS_{2}^{\lambda} + ab^{2}S_{2}^{\lambda}\otimes S_{2}^{\mu} 
 +abc S_{2}^{\lambda}\otimes S_{-2}^{\mu}+abd S_{2}^{\lambda} \otimes S_{0}^{\mu} \nonumber\\
 & +a^{2}c S_{-2}^{\lambda} + abcS_{-2}^{\lambda}\otimes S_{2}^{\mu} 
 +ac^{2}S_{-2}^{\lambda}\otimes S_{-2}^{\mu} + acd S_{-2}^{\lambda}\otimes S_{0}^{\mu}\nonumber 
+ a^{2}d S_{0}^{\lambda} +abd S_{0}^{\lambda}\otimes S_{2}^{\mu} \\
 & 
 +acd S_{0}^{\lambda}\otimes S_{-2}^{\mu} + ad^{2} S_{0}^{\lambda} \otimes S_{0}^{\mu} 
 + a^{2}c S_{2}^{\lambda \mu} +abcS_{2}^{\mu}\otimes S_{2}^{\lambda \mu} 
 +ac^{2}S_{-2}^{\mu}\otimes S_{2}^{\lambda \mu} +acd S_{0}^{\mu} \otimes S_{2}^{\lambda \mu} \nonumber\\
 & +abc S_{2}^{\lambda}\otimes S_{2}^{\lambda \mu} + b^{2}c S_{2}^{\lambda} \otimes S_{2}^{\mu} 
 \otimes S_{2}^{\lambda \mu} + bc^{2} S_{2}^{\lambda} \otimes S_{-2}^{\mu} 
 \otimes S_{2}^{\lambda \mu} +bcd S_{2}^{\lambda} \otimes S_{0}^{\mu} \otimes S_{2}^{\lambda \mu}\nonumber \\
 & +ac^{2} S_{-2}^{\lambda} \otimes S_{2}^{\lambda \mu} + bc^{2} S_{-2}^{\lambda} \otimes S_{2}^{\mu} 
 \otimes S_{2}^{\lambda \mu} + c^{3} S_{-2}^{\lambda} \otimes S_{-2}^{\mu} \otimes 
 S_{2}^{\lambda \mu} +c^{2}d S_{-2}^{\lambda} \otimes S_{0}^{\mu} \otimes S_{2}^{\lambda \mu}\nonumber \\
 & +acd S_{0}^{\lambda} \otimes S_{2}^{\lambda \mu} +bcd S_{0}^{\lambda} \otimes S_{2}^{\mu} 
 \otimes S_{2}^{\lambda \mu} +c^{2}d S_{0}^{\lambda} \otimes S_{-2}^{\mu} \otimes 
 S_{2}^{\lambda \mu} +cd^{2} S_{0}^{\lambda} \otimes S_{0}^{\mu} \otimes S_{2}^{\lambda \mu} \nonumber\\
 & +a^{2}b S_{-2}^{\lambda \mu} +ab^{2} S_{2}^{\mu} \otimes S_{-2}^{\lambda \mu} +abc S_{-2}^{\mu} 
 \otimes S_{-2}^{\lambda \mu} +abd S_{0}^{\mu} \otimes S_{-2}^{\lambda \mu} \nonumber\\
 & +ab^{2} S_{2}^{\lambda} \otimes S_{-2}^{\lambda \mu} + b^{3} S_{2}^{\lambda} \otimes S_{2}^{\mu} 
 \otimes S_{-2}^{\lambda \mu} +b^{2}c S_{2}^{\lambda}\otimes S_{-2}^{\mu} \otimes S_{-2}^{\lambda \mu} 
 + b^{2}d S_{2}^{\lambda} \otimes S_{0}^{\mu} \otimes S_{-2}^{\lambda \mu} \nonumber\\
 & +abc S_{-2}^{\lambda} \otimes S_{-2}^{\lambda \mu} +b^{2}c S_{-2}^{\lambda}\otimes S_{2}^{\mu} 
 \otimes S_{-2}^{\lambda \mu} +bc^{2} S_{-2}^{\lambda} \otimes S_{-2}^{\mu} \otimes S_{-2}^{\lambda \mu} 
 +bcd S_{-2}^{\lambda} \otimes S_{0}^{\mu} \otimes S_{-2}^{\lambda \mu} \nonumber\\
 & +abd S_{0}^{\lambda} \otimes S_{-2}^{\lambda \mu} +b^{2}d S_{0}^{\lambda} \otimes S_{2}^{\mu} \otimes 
 S_{-2}^{\lambda \mu} +bcd S_{0}^{\lambda} \otimes S_{-2}^{\mu} \otimes 
 S_{-2}^{\lambda \mu} +bd^{2} S_{0}^{\lambda} \otimes S_{0}^{\mu} \otimes S_{-2}^{\lambda \mu} \nonumber\\
 & + a^{2}d S_{0}^{\lambda \mu} + abd S_{2}^{\mu} \otimes S_{0}^{\lambda \mu} 
 +acd S_{-2}^{\mu} \otimes S_{0}^{\lambda \mu} +ad^{2} S_{0}^{\mu} \otimes S_{0}^{\lambda \mu} \nonumber\\
 & + abd S_{2}^{\lambda} \otimes S_{0}^{\lambda \mu} 
 +b^{2}d S_{2}^{\lambda} \otimes S_{2}^{\mu} \otimes S_{0}^{\lambda \mu} +bcd S_{2}^{\lambda} \otimes S_{-2}^{\mu} \otimes 
 S_{0}^{\lambda \mu} +bd^{2} S_{2}^{\lambda} \otimes S_{0}^{\mu} \otimes S_{0}^{\lambda \mu} \nonumber\\
 & +acd S_{-2}^{\lambda} \otimes S_{0}^{\lambda \mu} + bcd S_{-2}^{\lambda} \otimes S_{2}^{\mu} \otimes S_{0}^{\lambda \mu} 
 + c^{2}d S_{-2}^{\lambda} \otimes S_{-2}^{\mu} \otimes 
 S_{0}^{\lambda \mu} +cd^{2} S_{-2}^{\lambda} \otimes S_{0}^{\mu} \otimes S_{0}^{\lambda \mu}  \nonumber\\  
 & +ad^{2} S_{0}^{\lambda} \otimes S_{0}^{\lambda \mu} + bd^{2} S_{0}^{\lambda} \otimes S_{2}^{\mu} 
 \otimes S_{0}^{\lambda \mu} +cd^{2} S_{0}^{\lambda} \otimes S_{-2}^{\mu} 
 \otimes S_{0}^{\lambda \mu} +d^{3} S_{0}^{\lambda} \otimes S_{0}^{\mu} \otimes S_{0}^{\lambda \mu}, \nonumber 
 \end{align}
where $S_{\bullet}^{\lambda}=\pi_{1}^{\ast}(S_{\bullet}), S_{\bullet}^{\mu}=\pi_{2}^{\ast}(S_{\bullet})$ and $S_{\bullet}^{\lambda \mu}= \pi_{3}^{\ast}(S_{\bullet})$.
To obtain the E-polynomial of the total space, we need to substitute each representation by its associated E-polynomial, by Theorem \ref{thm:general-fibr}.

\begin{prop}
We have
\begin{itemize}
\item $e(T)=e(S_{2}^{\lambda}\otimes S_{2}^{\mu}\otimes S_{2}^{\lambda \mu})=e(S_{2}^{\lambda}\otimes S_{-2}^{\mu}\otimes S_{-2}^{\lambda \mu})=q^2-6q+9$,
\item $e(S_{0}^{\bullet})=e(S_{-2}^{\lambda}\otimes S_{-2}^{\mu} \otimes S_{-2}^{\lambda \mu})=e(S_{-2}^{\lambda}\otimes S_{2}^{\mu} \otimes S_{2}^{\lambda \mu})=-2q+6$,
\item $e(S_{\pm 2}^{\bullet})=e(S_{\pm 2}^{\bullet} \otimes S_{\pm 2}^{\bullet}) = -q+5$,
\end{itemize}
for $\bullet=\lambda,\mu,\lambda\mu$.
\end{prop}

\begin{proof}
Recall that the basis is ${B}'=\{ (\lambda,\mu) \in (\mathbb{C}^*)^{2} | \lambda, \mu, \lambda\mu \neq \pm 1 \} /\mathbb{Z}_{2}$.
We compute the E-polynomial of each representation case by case:
\begin{itemize}
\item $e(T)$. Using (\ref{eqn:+-}), we compute 
$e( (\mathbb{C}-\{0,\pm 1\})^{2}/\mathbb{Z}_{2})=(q-2)^2+1$, since $e(\mathbb{C}-\{0,\pm 1\})^+=q-2$ and 
$e(\mathbb{C}-\{0,\pm 1\})^-=-1$. Also $e( \{ (\lambda,\mu) \in (\mathbb{C}-\{0,\pm 1\})^{2} | \lambda\mu = \pm 1 \} /\mathbb{Z}_{2})=
2e((\mathbb{C}-\{0,\pm 1\}) /\mathbb{Z}_{2})=2(q-2)$. So $e(T)=e({B}')=q^2-4q+5-(2q-4)=q^2-6q+9$.
\item $e(S_{0}^{\lambda})$. Since $S_{0}^{\lambda}=\pi^{\ast}_{1}(S_{0})$, we need to compute the E-polynomial 
of the pullback bundle of the fibration $\mathbb{C}-  \lbrace 0,\pm 1 \rbrace \longrightarrow \mathbb{C}-  \lbrace \pm 2 \rbrace$ 
that maps $\lambda\mapsto\lambda + \lambda^{-1}$, and which has Hodge monodromy representation equal to $T+S_{0}$. 
The pullback bundle is
  $$
  \xymatrix{E_{S_{0}} \subset \bar E_{S_0}=(\mathbb{C}-  \lbrace 0,\pm 1 \rbrace)\times (\mathbb{C}-  \lbrace 0,\pm 1 \rbrace) \ar[d]^{p} \ar[r] & 
 \mathbb{C}-  \lbrace 0, \pm 1\rbrace \ar[d]^{g} \\ {B}' \subset \bar B'=(\mathbb{C}-  \lbrace 0,  \pm 1 \rbrace)\times 
 (\mathbb{C}-  \lbrace 0,\pm 1 \rbrace)/\mathbb{Z}_{2} \ar[r]^-{\pi_{1}} & \mathbb{C} -  \lbrace \pm 2 \rbrace \cong 
 (\mathbb{C}-  \lbrace 0,\pm 1 \rbrace)/\mathbb{Z}_{2}}
 $$
where $g(\lambda)= \lambda + \lambda^{-1}$ and $p$ is the quotient map. $E_{S_0}=p^{-1}(B')$, and $B'= \bar B' - \{(\lambda,\mu) |\lambda\mu=\pm 1\}$.
Then $e(\bar E_{S_{0}})=(q-3)^{2}$, $e(p^{-1} (\{(\lambda,\mu) |\lambda\mu=\pm 1\}))=e(\{(\lambda,\mu) \in (\mathbb{C}-  \lbrace 0,\pm 1 \rbrace)^2|\lambda\mu=\pm 1\})=2(q-3)$.
So $e(E_{S_0})=e(T+S_{0}^{\lambda})=(q-3)^{2}-2(q-3)=q^{2}-8q+15$. Finally,
 $$
 e(S_{0}^{\lambda})=e(T+S_{0}^{\lambda})-e(T)=-2q+6.   
 $$

The diagram still commutes if we change $\pi_{1}$ by $\pi_{2}$ or $\pi_{3}$, so we infer that 
$S_{0}^{\lambda} \cong S_{0}^{\mu} \cong S_{0}^{\lambda \mu}$ as local systems. 
This implies that certain reductions can be made: since $S_{2}^{\bullet} \otimes S_{-2}^{\bullet} \cong S_{0}^{\bullet}$, 
whenever $S_{0}^{\lambda}, S_{0}^{\mu}$ or $S_{0}^{\lambda \mu}$ appear in a tensor product we can switch one 
by another in order to simplify the expression. For example:
$$
S_{2}^{\lambda}\otimes S_{2}^{\mu} \otimes S_{0}^{\lambda \mu} \cong S_{2}^{\lambda} \otimes S_{2}^{\mu} \otimes S_{0}^{\mu} \cong S_{2}^{\lambda} \otimes S_{-2}^{\mu},
$$
$$
S_{0}^{\lambda}\otimes S_{-2}^{\mu}\otimes S_{0}^{\lambda \mu} \cong S_{0}^{\mu} \otimes S_{-2}^{\mu} \otimes S_{0}^{\mu} \cong S_{-2}^{\mu},
$$
and so on. In fact, we deduce from this that, for all representations $R$ that appear in (\ref{eqnlongrepresentations}), 
either $e(R)=e(S_{0}^{\lambda})$, $e(R)=e(S_{\pm 2}^{\lambda})$, $e(R)=e(S_{\pm 2}^{\lambda}\otimes S_{\pm 2}^{\mu})$ or 
$e(R)=e(S_{\pm 2}^{\lambda} \otimes S_{\pm 2}^{\mu} \otimes S_{\pm 2}^{\lambda \mu})$,
as any representation where $S_{0}^{\bullet}$ appears can be simplified to one of these expressions.

\item $e(S_{-2}^{\lambda})$. To obtain $e(S_{-2}^{\lambda})$, we take the pullback under the map
$\CC-\{\pm 2,0\} \to \CC -\{\pm 2\}$, $x\mapsto  x^2-2$, which ramifies at $-2$. The pullback fibration is
 $$
 \xymatrix{E_{S_{-2}} \subset \bar E_{S_{-2}}=\left\{ \frac{(y,\mu)}{(y,\mu)\sim (y^{-1},\mu^{-1})}\right\}  
 \ar[r] \ar[d]^{p} & x=y+y^{-1} \in 
 \mathbb{C} -  \lbrace \pm 2,0 \rbrace \ar[d]^{2:1} \\ B'\subset \bar{B}'=\lbrace [(\lambda=y^{2},\mu)]\rbrace 
 \ar[r] &  x^{2}-2=y^{2}+y^{-2}=\lambda+\lambda^{-1} \in \mathbb{C}-  \lbrace \pm 2 \rbrace}
 $$
where $p(y,\mu)=(y^{2},\mu)\in {B}'$, $y\in \mathbb{C}-  \lbrace 0, \pm 1, \pm \imat \rbrace$.  
Therefore $e(\bar E_{S_{-2}})=(q-3)(q-2)+2$. Substracting
the contribution corresponding to the two hyperbolas $\lbrace y^{2}\mu=\pm 1 \rbrace$, which has
E-polynomial $2(q-3)$, we get $e(E_{S_{-2}})=e(T+S_{-2}^{\lambda})=(q-3)(q-2)+2-2(q-3)=q^{2}-7q+14$, and
 $$
 e(S_{-2}^{\lambda})=e(T+S_{-2}^{\lambda})-e(T)=-q+5 .
 $$
The diagram for $e(S_{2}^{\lambda})$ is similar:
 $$
 \xymatrix{E_{S_{2}} \subset \bar E_{S_{2}}=\left\{ \frac{(y,\mu)}{(y,\mu)\sim (y^{-1},\mu^{-1})}\right\}  
 \ar[r] \ar[d]^{p} & x=y+y^{-1} \in 
 \mathbb{C} -  \lbrace \pm 2,0 \rbrace \ar[d]^{2:1} \\ B'\subset \bar{B}'=\lbrace [(\lambda=-y^{2},\mu)]\rbrace 
 \ar[r] &  2-x^2=y^{2}+y^{-2}=\lambda+\lambda^{-1} \in \mathbb{C}-  \lbrace \pm 2 \rbrace}
 $$
but now $p(y,\mu)=(-y^{2},\mu)$. Since $E_{S_{2}}\cong E_{S_{-2}}$, we obtain that $e(S_{2}^{\lambda})=
e(S_{-2}^{\lambda})$ and analogous computations yield that $e(S_{\pm 2}^{\lambda})=e(S_{\pm 2}^{\mu})=e(S_{\pm 2}^{\lambda\mu })$. 

\item $e(S_{-2}^{\lambda} \otimes S_{-2}^{\mu})$. To compute the E-polynomial of this representation, we can use the 
fibration with Hodge monodromy representation $(T+S_{-2}^{\lambda})\otimes (T+S_{-2}^{\mu})$, which 
corresponds to the fibered product of two copies of the pullback fibration $E_{S_{-2}}$, one with 
$\lambda=y^{2}$, the other with $\mu=z^{2}$. The total space is
 $$
  E_{S_{-2,-2}}\subset \bar E_{S_{-2,-2}} =
  \left\{ \left( \frac{(y,\mu)}{(y,\mu)\sim (y^{-1},\mu^{-1})}, \frac{(\lambda,z)}{(\lambda,z)\sim (\lambda^{-1},z^{-1})} \right) \right\}
  \longrightarrow B'\subset \bar B' =\{[\lambda,\mu]\},
 $$
where $\lambda=y^{2}, \mu=z^{2}$ and $(y,z)\sim (y^{-1},z^{-1}), y,z\in \mathbb{C}-  \lbrace 0, \pm 1, \pm \imat \rbrace$. 
The E-polynomial is $e( \bar E_{S_{-2,-2}}) =(q-3)^{2}+4$. Recall that we need to substract the contribution of the (now four) 
hyperbolas given by $\lbrace yz=\pm 1 \rbrace$  and $\lbrace yz=\pm \imat \rbrace$. The action $(y,z)\sim (y^{-1},z^{-1})$ acts on 
each of the first two hyperbolas giving a contribution of $2(q-3)$, whereas it interchanges the last two, which gives $q-5$. We get 
$e((T+S_{-2}^{\lambda})\otimes (T+S_{-2}^{\mu}))=e(E_{S_{-2,-2}})=(q-3)^{2}+4-2(q-3)-(q-5)=q^2-9q+24$ and
 \begin{align*}
 e(S_{-2}^{\lambda}\otimes S_{-2}^{\mu}) & = e((T+S_{-2}^{\lambda})\otimes (T+S_{-2}^{\mu}))
-e(T)-e(S_{-2}^{\lambda})-e(S_{-2}^{\mu}) \\ 
 & = q^{2}-9q+24-(q^{2}-6q+9)-2(-q+5)  =-q+5.
 \end{align*}
Computing the different fibered products of $E_{S_{2}}$ and $E_{S_{-2}}$ gives us 
$$
e(S_{2}^{\lambda}\otimes S_{-2}^{\mu})=e(S_{-2}^{\lambda}\otimes S_{2}^{\mu})=e(S_{2}^{\lambda}\otimes S_{2}^{\mu})=-q+5.
$$
Changing the projection does not alter the computations, so $e(S_{\pm 2}^{\bullet}\otimes S_{\pm 2}^{\bullet})=-q+5$.
\item $e(S_{-2}^{\lambda}\otimes S_{-2}^{\mu} \otimes S_{-2}^{\lambda \mu})$. To compute this E-polynomial, 
we can compute the E-polynomial of the representation $(T+S_{2}^{\lambda})\otimes (T+S_{2}^{\mu})\otimes (T+S_{2}^{\lambda \mu})$, 
which corresponds to the fibered product of three copies of $E_{S_{-2}}$. The total space of this fibered product is given by
 $$
 \left( \frac{(y,\mu)}{(y,\mu)\sim (y^{-1},\mu^{-1})}, \frac{(\lambda,z)}{(\lambda,z)\sim (\lambda^{-1},z^{-1})},
 \frac{(w,\lambda^{-1})}{(w,\lambda^{-1})\sim (w^{-1},\lambda)} \right),
 $$
where $\lambda=y^{2}, \mu=z^{2}, \lambda \mu =w^{2}$ and $yz=\pm w$, $y,z,w \in \mathbb{C} -  
\lbrace 0, \pm 1, \pm \imat \rbrace$ and $(y,z)\sim (y^{-1},z^{-1})$. 
Taking into account the two possible signs for $w$, and substracting the contribution from the hyperbolas in $\bar B'$, 
we obtain that the E-polynomial is $2(q^2-9q+24)$. This implies that
\begin{align*}
 e(S_{-2}^{\lambda}\otimes S_{-2}^{\mu} \otimes S_{-2}^{\lambda \mu})& = e((T+S_{-2}^{\lambda})\otimes (T+S_{-2}^{\mu}) 
\otimes (T+S_{-2}^{\lambda \mu})) - e(T)-3e(S_{-2}^{\lambda})-3e(S_{-2}^{\lambda}\otimes S_{-2}^{\mu})\\
 & = 2(q^{2}-9q+24)-(q^{2}-6q+9)-6(-q+5) = q^{2}-6q+9=e(T).
\end{align*}
An analogous computation gives the same polynomial for $e(S_{2}^{\lambda}\otimes 
S_{2}^{\mu} \otimes S_{-2}^{\lambda \mu})$ and cyclic permutations of signs.

\item $e(S_{-2}^{\lambda}\otimes S_{-2}^{\lambda}\otimes S_{2}^{\lambda})$. As we did in 
the previous case, to compute the E-polynomial it suffices to take the fibered product of two copies 
of $E_{S_{-2}}$ and $E_{S_{2}}$. The total space is again parametrized by
 $$
 \left( \frac{(y,\mu)}{(y,\mu)\sim (y^{-1},\mu^{-1})}, \frac{(\lambda,z)}{(\lambda,z)\sim (\lambda^{-1},z^{-1})},
 \frac{(w,\lambda^{-1})}{(w,\lambda^{-1})\sim (w^{-1},\lambda)} \right),
 $$
 where now $\lambda=y^{2}, \mu=z^{2}$ and $\lambda \mu=-w^{2}, y,z,w\in \mathbb{C}-\lbrace 0,\pm 1, \pm \imat \rbrace$. 
In particular, this implies that $yz=\imat w$ and $yz=-\imat w$, which gives two components that get identified under 
the $\mathbb{Z}_{2}$-action given by $(y,z,w)\sim (y^{-1},z^{-1},w^{-1})$. Therefore, the quotient is parametrized by 
$(y,z)\in (\mathbb{C} - \lbrace 0,\pm 1, \pm \imat \rbrace)^{2}$, which produces $(q-5)^{2}$. Substracting the four hyperbolas, 
we get that the E-polynomial of the fibered product is $(q-5)^{2}-4(q-5)=q^{2}-14q+45$. So
 \begin{align*}
  e(S_{-2}^{\lambda}\otimes S_{-2}^{\mu} \otimes S_{2}^{\lambda \mu})& = 
e((T+S_{-2}^{\lambda})\otimes (T+S_{-2}^{\mu}) \otimes (T+S_{2}^{\lambda \mu})) 
- e(T)-3e(S_{-2}^{\lambda})-3e(S_{-2}^{\lambda}\otimes S_{-2}^{\mu})\\
 & = (q^{2}-14q+45)-(q^{2}-6q+9)-6(-q+5) = -2(q-3)=e(S_{0}^{\bullet}).
 \end{align*}
\end{itemize}

\end{proof}
Substituting the values just obtained for the E-polynomial of every irreducible representation 
in (\ref{eqnlongrepresentations}), and the values $a=q^3, b=-3q, c=3q^2, d=-1$, we obtain the 
E-polynomial of the total fibration
 \begin{align} \label{eqn:eZ'}
 e(\overline{Z}/\mathbb{Z}_{2}) & =e(\overline{Z}')=e(R(\overline{Z}')) \\
 &=q^{11}-6q^{10}+54q^8-12q^7+189q^6-915q^5+666q^4+153q^3-81q^2-43q-6. \nonumber
 \end{align}

Using the formula (\ref{eqn:RX4-RX4Z2->eX4}), we get that 
\begin{align*}
 e(W_5) &= e(Z) =  (q^{2}-q) e(\overline{Z}') + q\, e(\overline{Z})  \\
 & = q^{13}-6q^{12}-2q^{11}+51q^{10}+12q^9+216q^8-573q^7-198q^6+696q^5-18q^4-125q^3-45q^2-9q.
\end{align*}

\subsection{Generic case} \label{subsec:generic}

Let us deal finally with the case $(t_{1},t_{2},t_{3})\not \in C$, $t_{i}\neq \pm 2$, which corresponds to the open 
subset $U$ in $(\mathbb{C}^{\ast})^{3}$ defined by those representations whose $\xi_1,\xi_2,\xi_3$ are diagonalizable
and do not share an eigenvector. Choosing a basis that diagonalizes $\xi_{1}$, note that, if we write $\xi_{2}=\begin{pmatrix} a & b \\ c & d \end{pmatrix}$, then
 $$
 \begin{pmatrix} \lambda_1 & 0 \\ 0 & \lambda_1^{-1} \end{pmatrix} 
 \begin{pmatrix} a & b \\ c & d \end{pmatrix} \xi_{3} = \begin{pmatrix} -1 & 0 \\ 0 & -1 \end{pmatrix},
 $$
so that
 $$
 \xi_{3}= \begin{pmatrix}  -\lambda_1^{-1}d & \lambda_1 b \\ \lambda_1^{-1}c & -\lambda_1 a \end{pmatrix} 
 $$
and $bc\neq 0$ (see Lemma \ref{lem:C}). Conjugating by a diagonal matrix, we can assume that $b=1$. As 
$t_2=a+d$ and $t_3=- \lambda_1a-\lambda_1^{-1}d$, we have that 
 $a,d$ are determined by the values of $(t_{2},t_{3})$; and $c$ is determined by the equation 
$\det \xi_{2}=ad-bc=1$. We see that for fixed $(\lambda_1,t_{2},t_{3})$, 
$a,b,c,d$ are fully determined, and so are $\xi_{2},\xi_{3}$. 

Consider the $\ZZ_2\x\ZZ_2\x\ZZ_2$-cover given by $(\lambda_1,\lambda_{2},\lambda_{3}) \mapsto (t_{1},t_{2},t_{3})$ over
$(\CC-\{0,\pm 1\})^3$. Let $\bar E$ be the pull-back fibration. 
The fiber over $(\lambda_1,\lambda_{2},\lambda_{3})$ is isomorphic to
 $$
\overline{X}_{4,\lambda_1} \times \overline{X}_{4,\lambda_2} \times \overline{X}_{4, \lambda_3} .
 $$
Let $(A_1,B_1,A_2,B_2,A_3,B_3) \in \bar E$. Then $[A_1,B_1]=\xi_1$, $[A_2,B_2]=\xi_2$, $[A_3,B_3]=\xi_3$, where $a=a(\lambda_1,\lambda_2,\lambda_3),
b=1,c=c(\lambda_1,\lambda_2,\lambda_3), d=d(\lambda_1,\lambda_2,\lambda_3)$. Take $Q=Q(\lambda_1,\lambda_2,\lambda_3)$,
$S=S(\lambda_1,\lambda_2,\lambda_3)$ matrices such that $Q^{-1}\xi_2 Q=  \begin{pmatrix} \lambda_2 & 0 \\ 0 & \lambda_2^{-1} \end{pmatrix}$
and $S^{-1}\xi_3 S=  \begin{pmatrix} \lambda_3 & 0 \\ 0 & \lambda_3^{-1} \end{pmatrix}$.
Then $\Theta (A_1,B_1,A_2,B_2,A_3,B_3) = (A_1,B_1,Q^{-1} A_2Q,Q^{-1} B_2Q,S^{-1} A_3S,S^{-1} B_3S)$
identifies $\bar E$ with the subset of $\overline{X}_4 \x \overline{X}_4 \x \overline{X}_4$
where $(t_1,t_2,t_3)\not\in C$. The second and third copies of $\ZZ_2$ act as the standard $\ZZ_2$-action on 
the second and third copies of $\overline{X}_4$, respectively. 
The action of the first copy of $\ZZ_2$ is more delicate: it acts as conjugation by 
$P_0= \begin{pmatrix} 0 & 1 \\ c & 0 \end{pmatrix}$,
sending $\xi_1= \begin{pmatrix} \lambda_1 & 0 \\ 0 & \lambda_1^{-1} \end{pmatrix}
\mapsto  \begin{pmatrix} \lambda_1^{-1} & 0 \\ 0 & \lambda_1 \end{pmatrix}$,
$\xi_2= \begin{pmatrix} a& 1 \\ c & d \end{pmatrix} \mapsto  \begin{pmatrix} d & 1 \\ c & a \end{pmatrix}$.
Then under the isomorphism $\Theta$, it acts by conjugation by $Q^{-1}P_0  Q$  (resp.\ $S^{-1}P_0  S$) on the
second (resp.\ third) factor of $\overline{X}_4$. This matrix is easily computed to be
diagonal, so the action is (homologically) trivial.

The conclusion is that $\bar E/ \ZZ_2\x\ZZ_2\x\ZZ_2 $ is isomorphic to the open set of 
$\overline{X}_{4}/\mathbb{Z}_{2}\times \overline{X}_{4}/\mathbb{Z}_{2} \times \overline{X}_{4}/\mathbb{Z}_{2}$, 
where $(t_1,t_2,t_3)\not\in C$.

Now we need to compute the E-polynomial of 
$T=\left(\overline{X}_{4}/\mathbb{Z}_{2}\times \overline{X}_{4}/\mathbb{Z}_{2} \times
\overline{X}_{4}/\mathbb{Z}_{2} \right) \cap \{(t_1,t_2,t_3) \in C\}$. Here we can
parametrize $C \cong \{(\lambda,\mu) \in (\CC^*)^3 | \lambda,\mu,\lambda\mu\neq \pm 1\}/\ZZ_2$. 
The fiber over $(\lambda, \mu)$ is isomorphic to
$\overline{X}_{4,\lambda} \times \overline{X}_{4,\mu} \times \overline{X}_{4,-\lambda^{-1}\mu^{-1}}$,
hence this space is isomorphic to $\overline{Z}'$, studied in (\ref{eqn:barZ}). Using (\ref{eqn:eZ'}), we get
 \begin{align*}
 e(W_{6}) =&  \, e(\pgl)(e(\overline{X}_{4}/\mathbb{Z}_{2})^{3}- e(\overline{Z}')) \\
 = &\,  q^{15}-7q^{14}+8q^{13}+44q^{12}-78q^{11}-136q^{10}-153q^9 \\
  & +1149q^8-450q^7-1263q^6+798q^5+265q^4-119q^3-52q^2-7q.
 \end{align*}

\subsection{Final result} \label{subsec:finalresult}

If we add all the E-polynomials of the different strata, we get
 \begin{align*}
 e(W) & = e(W_1) +e(W_2) +e(W_3) +e(W_4) +e(W_5) +e(W_6) \\
 &= q^{15}-5q^{13}+10q^{11}-252q^{10}-20q^9+20q^7+252q^6-10q^5+5q^3-q.
 \end{align*}

Then
 $$
e(\cM^1)=e(W)/e(\pgl)=q^{12}-4q^{10}+6q^8-252q^7-14q^6-252q^5+6q^4-4q^2+1.
 $$
This agrees with the result in \cite{mereb:2010}, obtained by arithmetic methods.

\section{Hodge monodromy representation for the genus $2$ character variety} \label{sec:charvar-g=2}

We introduce the following sets associated to the representations of a genus $2$ complex curve,
and give the E-polynomials computed in  \cite{lomune}:
\begin{itemize}
\item $Y_0:=\{(A_1,B_1,A_2,B_2)\in \SL(2,\CC)^4\, |\, [A_1,B_1][A_2,B_2]=\Id\}$. Then
$e(Y_0)= q^9+q^8+12q^7+2q^6-3q^4-12q^3-q$, by \cite[Section 8.1]{lomune}.
\item $Y_1:=\{(A_1,B_1,A_2,B_2)\in \SL(2,\CC)^4\,|\, [A_1,B_1][A_2,B_2]=-\Id\}$. Then
 $e(Y_1)= q^9-3q^7-30q^6+30q^4+3q^3-q$, by \cite[Section 9]{lomune}.
\item $\overline{Y}_2:=\left\{(A_1,B_1,A_2,B_2)\in \SL(2,\CC)^4\,|\, [A_1,B_1][A_2,B_2]=J_+=\small{\left(
    \begin{array}{cc}
      1 & 1\\
      0 & 1\\
    \end{array}
  \right)}\right\}$, 
$e(\overline{Y}_2)= q^9-3q^7-4q^6-39q^5-4q^4-15q^3$, by \cite[Section 11]{lomune}.
  \item $\overline{Y}_3:=\left\{(A_1,B_1,A_2,B_2)\in \SL(2,\CC)^4\, |\, [A_1,B_1][A_2,B_2]=J_-=\small{\left(
    \begin{array}{cc}
      -1 & 1\\
      0 & -1\\
    \end{array}
  \right)}\right\}$. Then
$e(\overline{Y}_3)= q^9-3q^7+15q^6+6q^5+45q^4$, by \cite[Section 12]{lomune}.
\item $\overline{Y}_{4,\lambda}:=\left\{(A_1,B_1,A_2,B_2) \, |\, [A_1,B_1][A_2,B_2]=\small{\left(
    \begin{array}{cc}
      \lambda & 0\\
      0 & \lambda^{-1}\\
    \end{array}
  \right)} \right\}
  $,
for $\lambda\neq 0,\pm 1$. Then
$e(\overline{Y}_{4,\lambda})= q^9-3q^7+15q^6-39q^5+39q^4-15q^3+3q^2-1$, 
by \cite[Section 10]{lomune}.
\end{itemize}

Let $\overline{Y}_4:=\left\{(A_1,B_1,A_2,B_2,\lambda)\in \SL(2,\CC)^4 \x \CC^*\, |\, [A_1,B_1][A_2,B_2]=\small{\left(
    \begin{array}{cc}
      \lambda & 0\\
      0 & \lambda^{-1}\\
    \end{array}
  \right)},\, \lambda\neq 0,\pm 1\right\}$. We have a fibration 
 $$ 
 \overline{Y}_4 \too \CC-\{0,\pm 1\}.
 $$
If we take the quotient by the $\ZZ_2$-action there is another fibration
 $$ 
 \overline{Y}_4/\ZZ_2 \too \CC-\{\pm 2\}.
 $$
We are interested in the Hodge monodromy representations $R(\overline{Y}_4)$ and $R(\overline{Y}_4/\ZZ_2)$.

\begin{prop} \label{prop:RY4}
  $R(\overline{Y}_4)=(q^9-3q^7+6q^5-6q^4+3q^2-1)T + (15q^6-45q^5+45q^4-15q^3)N$.
\end{prop}

\begin{proof}
We follow the stratification $\overline{Y}_{4,\lambda_{0}}=\bigsqcup_{i=1}^7 Z_{i}$ given in \cite[Section 10] {lomune}, and 
study the behaviour of each stratum when $\lambda$ varies in $ \CC - \{ 0, \pm 1 \}$ to obtain the Hodge monodromy representation of $\overline{Y}_{4}$. 
Let $\xi=\begin{pmatrix} \lambda & 0 \\ 0 & \lambda^{-1} \end{pmatrix}$. As in  \cite[Section 10]{lomune}, we write
 $$
 \nu=[B_{2},A_{2}]= \begin{pmatrix} a & b \\ c & d \end{pmatrix}, \quad \delta=[A_{1},B_{1}] =
 \xi\nu = \begin{pmatrix} \lambda a & \lambda b \\ \lambda^{-1}c & \lambda^{-1}d \end{pmatrix},
 $$
and $t_{1}=\tr \nu$, $t_{2}=\tr \delta$. Note that every $(t_{1},t_{2}, \lambda)$ determines $a,d$ by
$$
a=\frac{t_2-\lambda^{-1}t_1}{\lambda-\lambda^{-1}}, \quad d=\frac{\lambda t_1-t_2}{\lambda-\lambda^{-1}}.
$$
Then $bc=ad-1$.

We look at the strata:
\begin{itemize}
\item $Z_{1}$, corresponding to $t_{1}=\pm 2, t_{2}=\pm 2$. In this case, both $\nu,\delta$ are of Jordan type. 
If we take the basis given by $\{ u_{1},u_{2} \}$, where $u_{1}$ is an eigenvector for $\nu$ and $u_{2}$ an eigenvector for $\delta$, then
$$
\nu=\begin{pmatrix} 1 & x \\ 0 & 1\end{pmatrix}, \quad \delta=\begin{pmatrix} 1 & 0 \\ y & 1\end{pmatrix}
$$
for certain $x,y\in \CC^*$. Now, since $\delta\nu^{-1}=\begin{pmatrix} \lambda & 0 \\ 0 & \lambda^{-1} \end{pmatrix}$, we obtain that 
$\lambda+\lambda^{-1}=2-xy$. We can fix $x=1$ by rescaling the basis, so $y$ is fixed and there is no monodromy around the origin. Therefore
\begin{align*}
R(Z_{1}) & =e(Z_{1})T=(q-1)(e(\overline{X}_{2})+e(\overline{X}_{3}))^{2}T \\
& = (4q^7-15q^5+5q^4+15q^3-9q^2)T,
\end{align*}
where $T$ is the trivial representation.

\item $Z_{2}$, corresponding to $t_{1}=2, t_{2}=\lambda+\lambda^{-1}$ and $t_{2}=2, t_{1}=\lambda+\lambda^{-1}$. We focus on the first case. 
In this situation, $bc=0$, so there are three possibilities: either $b=c=0$ (in which case $\nu=\Id$) or $b=0,c\neq 0$ or $b\neq 0,c=0$ (in either case 
there is a parameter  in $\CC^*$ and $\nu$ is of Jordan type). In every situation, $\nu$ has trivial monodromy, whereas 
$\delta\sim \begin{pmatrix} \lambda & 0 \\ 0 & \lambda^{-1} \end{pmatrix}$. This contributes $R(\overline{X}_{4})$. Therefore
\begin{align*}
R(Z_{2}) & =  2(e(X_{0})+2(q-1)e(\overline{X}_{2}))R(\overline{X}_{4}) \\
& = (6q^7-4q^6-6q^5-2q^4+4q^3+6q^2-4q)T + (18q^6-30q^5-6q^4+30q^3-12q^2)N.
\end{align*}

\item $Z_{3}$, given by $t_{1}=-2,t_{2}=-\lambda-\lambda^{-1}$ and $t_{2}=-2,t_{1}=-\lambda-\lambda^{-1}$. This is analogous to the previous case, so
\begin{align*}
R(Z_{3}) & =  2(e(X_{1})+2(q-1)e(\overline{X}_{3}))R(\overline{X}_{4}) \\
& = (4q^7+10q^6-12q^5-6q^4-10q^3+12q^2+2q)T+(12q^6+18q^5-66q^4+30q^3+6q^2)N.
\end{align*}

\item $Z_{4}$, defined by $t_{1}=2, t_{2}\neq \pm 2, \lambda+\lambda^{-1}$ and $t_{2}=2, t_{1}\neq \pm 2, \lambda+\lambda^{-1}$. 
Both cases are similar, so we do the first case.
For each $\lambda$, $(t_{1},t_{2})$ move in a punctured line $\{ (t_{1},t_{2}) \mid t_{1}=2, t_{2}\neq \pm 2, \lambda+\lambda^{-1} \}$, 
where $\nu$ is of Jordan form and $\delta$ is of diagonal type, with trace $t_{2}$. Both families can be trivialized, giving a contribution of 
$e(\overline{X}_{2})$ and $e(\overline{X}_{4}/\mathbb{Z}_{2})$. The missing fiber $\overline{X}_{4,\lambda}$ over $\lambda+\lambda^{-1}$, 
which needs to be removed, has monodromy representation $R(\overline{X}_{4})$ as $\lambda$ varies. Therefore
\begin{align*}
R(Z_{4}) & = 2(q-1)e(\overline{X}_{2})(e(\overline{X}_{4}/\mathbb{Z}_{2})T-R(\overline{X}_{4})) \\
& = (2q^8-12q^7+10q^6+36q^5-26q^4-36q^3+14q^2+12q)T+(-6q^6+24q^5-12q^4-24q^3+18q^2)N.
\end{align*}
The factor $(q-1)$ corresponds to the fact that now $bc\neq 0$, so there is the extra freedom given by $\CC^*$.
 
\item $Z_{5}$, defined by $t_{1}=-2, t_{2}\neq \pm 2,- \lambda-\lambda^{-1}$ and $t_{2}=-2, t_{1}\neq \pm 2, 
-\lambda-\lambda^{-1}$. Similarly to $Z_{4}$, we obtain
\begin{align*}
R(Z_{5}) & =  2(q-1)e(\overline{X}_{3})(e(\overline{X}_{4}/\mathbb{Z}_{2})T-R(\overline{X}_{4})) \\
& =  (2q^8-2q^7-24q^6+12q^5+34q^4-10q^3-12q^2)T+(-6q^6-6q^5+30q^4-18q^3)N.
\end{align*}

\item $Z_{6}$. This stratum corresponds to the set $\{ (t_{1},t_{2})\mid t_{1},t_{2}\neq \pm 2, ad=1 \}$, 
which is a hyperbola $H_\lambda$ for every $\lambda$. 
Since $bc=0$, we get a contribution of $2q-1$, arising from the disjoint cases $b=c=0$; $b=0,c\neq 0$; and $b\neq 0, c=0$. Parametrizing 
$H_\lambda$ by a parameter $\mu \in \CC^* -\{ \pm 1, \pm \lambda^{-1} \}$ as in \cite[Section 10]{lomune}, 
we obtain a fibration over $\CC^* -\{ \pm 1, \pm \lambda^{-1} \}$ 
whose fiber over $\mu$ is $\overline{X}_{4,\mu}\times \overline{X}_{4,\lambda\mu}$, for fixed $\lambda$. 
When $\lambda$ varies over $\mathbb{C}^{*} -\{ \pm 1 \}$, 
note that we can extend the local system trivially to the cases $\lambda,\mu = \pm 1$. This extension 
can be regarded as a local system over the set of $(\lambda,\mu) \in \CC^*\times \CC^*$
 $$
 \overline{Z}_{6}=\overline{X}_{4}\times m^{*}\overline{X}_{4} \longrightarrow \CC^*\times\CC^*\, , 
 $$
where $m: \CC^*\times \CC^* \rightarrow \CC^*$ maps $(\lambda,\mu)\mapsto \lambda\mu$. 
The Hodge monodromy representation of $\overline{Z}_{6}$ belongs to $R(\zz)[q]$ (with generators 
$N_{1},N_{2}$ denoting the  representation which is not trivial over the generator of the 
fundamental group of the first and second copies of $\CC^*$ respectively, and
$N_{12}=N_1\otimes N_2$). Since $R(\overline{X}_{4})=(q^3-1)T+(3q^2-3q)N$, we get
\begin{align*}
R_{ \CC^*\times\CC^*} (\overline{Z}_{6}) & = ((q^{3}-1)T+(3q^{2}-3q)N_{2}) \otimes ((q^{3}-1)T+(3q^2-3q)N_{12}) \\
& = (q^{3}-1)^2 T +(3q^{2}-3q)^{2}N_{1} + (3q^{2}-3q)(q^{3}-1)N_{2}+(3q^{2}-3q)(q^{3}-1)N_{12}.
\end{align*}
We write this as $R_{ \CC^*\times\CC^*}(\overline{Z}_{6})=aT+bN_{1}+cN_{2}+dN_{12}$.
To obtain the Hodge monodromy representation over $\lambda \in \CC^*$, we use the projection 
$\pi_{1}:\CC^*\times\CC^* \rightarrow \CC^*$, $(\lambda,\mu)\mapsto \lambda$. 
Then $T\mapsto e(T) T$, $N_2 \mapsto e(N_2)T$, $N_1\mapsto e(T) N$, 
$N_{12} \mapsto e(N_2)N$ for the representations. Using that $e(T)=q-1$ and $e(N_2)=0$ and substracting the contribution from the sets $\mu=\pm 1, \pm \lambda^{-1}$, which yield $4e(\overline{X}_{4,\lambda})R(\overline{X}_{4})$, we get 
\begin{align*}
R(\overline{Z}_{6}) & =a\,e(T)T+b\,e(T)N+c\,e(N_{2})T+d\,e(N_{2}) N -4e(\overline{X}_{4,\lambda})R(\overline{X}_{4}) \\
& =    (q^7- 5 q^6 - 12 q^5+ 10 q^4 + 10 q^3 + 12 q^2 - 11 q-5)T + ( - 3 q^5 - 51 q^4 + 99 q^3- 33 q^2-12 q )N
\end{align*}
and
\begin{align*}
R(Z_6) =& (2q-1) R(\overline{Z}_{6}) \\
=& (2q^8-11q^7-19q^6+32q^5+10q^4+14q^3-34q^2+q+5)T \\
 &+(-6q^6-99q^5+249q^4-165q^3+9q^2+12q)N.
\end{align*}

\item $Z_{7}$, corresponding to the open stratum given by the set of $(t_{1},t_{2})$ such that $t_{i} \neq \pm 2$, 
$i=1,2$ and $(t_{1},t_{2})\not\in H_\lambda$. If we forget about the condition $(t_{1},t_{2})\in H_\lambda$, $Z_{7}$ is a fibration over 
$(t_{1},t_{2})$ with fiber isomorphic to $\overline{X}_{4,\mu_{1}}\times \overline{X}_{4,\mu_{2}}$, $t_{i}=\mu_i+\mu^{-1}_i$, $i=1,2$. 
Its monodromy is trivial, as the local system is trivial when $\lambda$ varies. The contribution over $H_\lambda$, already computed in the previous
stratum, is $R(\overline{Z}_6)$. So we get
\begin{align*}
R(Z_{7}) =& (q-1)(e(\overline{X}_{4}/\mathbb{Z}_{2})^2 T-R(\overline{Z}_{6})) \\
=&  (q^9-6q^8+8q^7+27q^6-41q^5-21q^4+23q^3+26q^2-11q-6)T \\
 &+ (3q^6+48q^5-150q^4+132q^3-21q^2-12q)N.
\end{align*}
\end{itemize}

Adding all pieces, we get
$$
R(\overline{Y}_{4})=(q^9-3q^7+6q^5-6q^4+3q^2-1)T + (15q^6-45q^5+45q^4-15q^3)N.
$$
\end{proof}

Dividing by $q-1$, we get the formula (\ref{eqn:RcM}).

We want to compute the Hodge monodromy representation of $\overline{Y}_{4}/\mathbb{Z}_{2}$. We have the following.

\begin{lem}
 The Hodge monodromy representation $R(\overline{Y}_{4}/\mathbb{Z}_{2})$ is of the form
$R(\overline{Y}_{4}/\mathbb{Z}_{2})= aT+bS_{2}+cS_{-2}+dS_{0}$, for some polynomials $a,b,c,d\in \ZZ[q]$.
\end{lem}

\begin{proof}
 The Hodge monodromy representation $R(\overline{Y}_{4}/\mathbb{Z}_{2})$ lies in the representation ring of the
 fundamental group of $\CC-\{\pm 2\}$. Under the double cover $\CC-\{0,\pm 1\} \to \CC-\{ \pm 2\}$, it reduces
to $R(\overline{Y}_{4})$. By Proposition \ref{prop:RY4},  $R(\overline{Y}_{4})$ is of order $2$. Hence
$R(\overline{Y}_{4}/\mathbb{Z}_{2})$ has only monodromy of order $2$ over the loops $\gamma_{\pm 2}$ 
around the points $\pm 2$. This is the statement of the lemma.
\end{proof}

To compute $a,b,c,d\in \ZZ[q]$, we  compute the E-polynomial of the twisted $\SL(2,\CC)$-character variety (\ref{eqn:twisted}) in
another way. Stratify 
 $$
 W:=\{(A_1,B_1,A_2,B_2,A_3,B_3)\in \SL(2,\CC)^{6}\, | \, [A_1,B_1][A_2,B_2]=-[A_3,B_3]\}
 $$
as follows:
\begin{itemize}
\item $W_{0}= \lbrace (A_1,B_1,A_2,B_2,A_3,B_3) \mid [A_1,B_1][A_2,B_2]=-[B_3,A_3]=\Id \rbrace$. Then 
 $$
 e(W_0)=e(Y_0)e(X_1)=q^{12}+q^{11}+11q^{10}+q^9-12q^8-5q^7-12q^6+3q^5+11q^4+q^2.
 $$
\item $W_{1}= \lbrace (A_1,B_1,A_2,B_2,A_3,B_3)  \mid [A_1,B_1][A_2,B_2]=-[B_3,A_3]=-\Id \rbrace$. Then 
 $$
 e(W_1)=e(Y_1)e(X_0)=q^{13}+4q^{12}-4q^{11}-46q^{10}-117q^9+72q^8+243q^7-18q^6-124q^5-16q^4+q^3+4q^2.
 $$
\item $W_{2}= \lbrace (A_1,B_1,A_2,B_2,A_3,B_3) \mid [A_1,B_1][A_2,B_2]=-[B_3,A_3]\sim J_{+} \rbrace$. Then 
 $$
 e(W_2)=e(\pgl/U) e(\overline{Y}_2)e(\overline{X}_3)=q^{14}+3q^{13}-4q^{12}-16q^{11}-48q^{10}-108q^9+24q^8+76q^7+27q^6+45q^5.
 $$
\item $W_{3}=\lbrace (A_1,B_1,A_2,B_2,A_3,B_3) \mid [A_1,B_1][A_2,B_2]=-[B_3,A_3]\sim J_{-} \rbrace$. Then 
 $$
 e(W_3)=e(\pgl/U) e(\overline{Y}_3)e(\overline{X}_2)=q^{14}-2q^{13}-7q^{12}+23q^{11}-9q^{10}-33q^9-93q^8-123q^7+108q^6+135q^5
 $$
\item $W_{4}= \lbrace (A_1,B_1,A_2,B_2,A_3,B_3) \mid [A_1,B_1][A_2,B_2]=-[B_3,A_3]\sim \begin{pmatrix} \lambda & 0 \\ 0 & \lambda^{-1} \end{pmatrix}, \lambda \neq 0, \pm 1  \rbrace$.
\item $\overline{W}_{4}= \lbrace (A_1,B_1,A_2,B_2,A_3,B_3,\lambda) \mid [A_1,B_1][A_2,B_2]=-[B_3,A_3]=\begin{pmatrix} \lambda & 0 \\ 0 & \lambda^{-1} \end{pmatrix}, \lambda \neq 0, \pm 1  \rbrace$.
\item $\overline{W}_{4,\lambda}= \lbrace (A_1,B_1,A_2,B_2,A_3,B_3)  \mid [A_1,B_1][A_2,B_2]=-[B_3,A_3]=\begin{pmatrix} \lambda & 0 \\ 0 & \lambda^{-1} \end{pmatrix}\rbrace$, where 
$\lambda \neq 0, \pm 1$.
\end{itemize}

Using the formula in Section \ref{subsec:finalresult},
 \begin{align*} 
  e(W_4) &= e(W)-e(W_0)-e(W_1)-e(W_2)-e(W_3) \\
 &= q^{15}-2q^{14}-7q^{13}+6q^{12}+6q^{11}-160q^{10}+237q^9+9q^8-171q^7+147q^6-69q^5+5q^4+4q^3-5q^2-q.
 \end{align*}

For the last stratum, note that:
 \begin{equation*} 
 \overline{W}_{4,\lambda}=\overline{Y}_{4,\lambda}\times \overline{X}_{4,-\lambda}\, .
 \end{equation*}
So $ R(\overline{W}_4/\ZZ_2)=R(\overline{Y}_4/\ZZ_2)\otimes R(\tau^*\overline{X}_4/\ZZ_2)$, where
$\tau: \CC-\{\pm 2\}\to \CC-\{\pm 2\}$, $\tau(x)=-x$. Then
 \begin{eqnarray*} 
 R(\overline{W}_4/\ZZ_2)&=&R(\overline{Y}_4/\ZZ_2)\otimes \tau^*R(\overline{X}_4/\ZZ_2) \\
  &=& (aT +b S_2+cS_{-2}+d S_0) \otimes (q^3T+3q^2 S_2 -3q S_{-2} -S_0) \\
  & =& (q^3a+3q^2b-3qc-d)T + (3q^2a+q^3b-c-3qd)S_{2} \\
  & & + (-3qa-b+q^3c+3q^2d)S_{-2} + (-a-3qb+3q^2c+q^3d)S_{0} \\
 &=& a'T +b' S_2+c'S_{-2}+d' S_0
\end{eqnarray*}
Using (\ref{eqn:RX4-RX4Z2->eX4}), we get
 \begin{align*}
e(W_4) &= (q^2-q) e(\overline{W}_4/\ZZ_2) +q\, e(\overline{W}_4) \\
 &= (q^2-q) ((q-2)a'-(b'+c'+d'))+ q((q-3)(a'+d')-2(b'+c')) ,
 \end{align*}
 which gives us the equation
 \begin{align} \label{eqn:HMRY4z2ec1}
 e(W_4) = & \, a(q^6-2q^5-4q^4+3q^2+2q) + b(2q^5-7q^4-3q^3+7q^2+q)S \nonumber \\
  & + c(-q^5-4q^4+4q^2+q) + d(-5q^4-q^3+5q^2+q).
 \end{align}
 
We can obtain another equation if we recall that
$$
Y_{4}:= \{ (A_{1},B_{1},A_{2},B_{2}) \mid [A_{1},B_{1}][A_{2},B_{2}]\sim 
\begin{pmatrix} \lambda & 0 \\ 0 & \lambda^{-1} \end{pmatrix}, \lambda \neq 0, \pm 1\}.
$$
Using that $\SL(2,\mathbb{C})^{4}=\bigsqcup_{i=0}^{4} Y_{i}$, we obtain that:
\begin{align*}
e(Y_{4}) & = e(\SL(2,\mathbb{C})^{4})-e(Y_{0})-e(Y_{1})-e(Y_{2})-e(Y_{3}) \\
& = q^{12}-2q^{11}-4q^{10}+6q^9-6q^8+18q^7-6q^6-18q^5+15q^4-6q^3+2q.
\end{align*}
But the E-polynomial of $Y_{4}$ can again be obtained using the Hodge monodromy representation 
$R(\overline{Y}_{4}/\mathbb{Z}_{2})$, using (\ref{eqn:RX4-RX4Z2->eX4}),
\begin{align} \label{eqn:HMRY4z2ec2}
e(Y_{4}) &= (q^{2}-q)e(\overline{Y}_{4}/\ZZ_2) +q\,e(\overline{Y}_{4}) \nonumber \\
& = (q^2-q) ((q-2)a-(b+c+d)) +q((q-3)(a+d)-2(b+c)) \\
& = (q^3-2q^2-q)a -(q^2+q)(b+c) -2qd . \nonumber
\end{align}
Finally, two more equations arise from the E-polynomial of the fiber of $\overline{Y}_{4}/\ZZ_2 \longrightarrow \CC-\{ \pm 2 \}$,
\begin{equation} \label{eqn:HMRY4z2ec3}
e(Y_{4,\lambda})=a+b+c+d,
\end{equation}
and from the Hodge monodromy representation $R(\overline{Y}_4)=(q^9-3q^7+6q^5-6q^4+3q^2-1)T 
+ (15q^6-45q^5+45q^4-15q^3)N$ given in Proposition \ref{prop:RY4}. Since $R(\overline{Y}_{4})=(a+d)T+(b+c)N$, we get the equation
\begin{equation} \label{eqn:HMRY4z2ec4}
  a+d=q^9-3q^7+6q^5-6q^4+3q^2-1.
\end{equation}

From equations \eqref{eqn:HMRY4z2ec1}, \eqref{eqn:HMRY4z2ec2}, \eqref{eqn:HMRY4z2ec3} and (\ref{eqn:HMRY4z2ec4}), we find 
 \begin{align*}
 a&= q^9 - 3q^7 + 6q^5\\
 b&= -45q^5 - 15q^3\\
 c&= 15q^6 + 45q^4\\
 d&=-6q^4 + 3q^2-1.
 \end{align*}
We have proved:
\begin{prop} \label{RY4z2}
$R(\overline{Y}_{4}/\ZZ_2)=(q^9-3q^7+6q^5)T-(45q^5+15q^3)S_{2}+(15q^6+45q^4)S_{-2}+(-6q^4+3q^2-1)S_{0}$.
\end{prop}

\section{E-polynomial of the character variety of genus $3$} \label{sec::charvar}

Let $\mathcal{M}=\cM(\SL(2,\CC))$ be the character variety of a genus $3$ complex curve $X$, i.e, the moduli 
space of semisimple representations of its fundamental group into $\sldos$. It can be defined as the space
 $$
 \mathcal{M} = V //\pgl,
 $$
where
 $$
 V= \lbrace (A_1,B_1,A_2,B_2,A_3,B_3) \in \sldos^{6} \mid [A_1,B_1][A_2,B_2][A_3,B_3]=\Id \rbrace.
 $$
We stratify $V$ as follows:
\begin{itemize}
\item $V_{0}= \lbrace (A_1,B_1,A_2,B_2,A_3,B_3)  \mid [A_1,B_1][A_2,B_2]=[B_3,A_3]=\Id \rbrace$.
Then
 $$
 e(V_0)=e(Y_0)e(X_0)= q^{13}+5q^{12}+15q^{11}+45q^{10}-8q^9-53q^8-32q^7-45q^6+23q^5+44q^4+q^3+4q^2.
 $$
\item $V_{1}= \lbrace (A_1,B_1,A_2,B_2,A_3,B_3)  \mid [A_1,B_1][A_2,B_2]=[B_3,A_3]=-\Id \rbrace$. Then
 $$
 e(V_1)=e(Y_1)e(X_1)= q^{12}-4q^{10}-30q^9+3q^8+60q^7+3q^6-30q^5-4q^4+q^2.
 $$
\item $V_{2}= \lbrace (A_1,B_1,A_2,B_2,A_3,B_3)  \mid [A_1,B_1][A_2,B_2]=[B_3,A_3]\sim J_{+} \rbrace$.
 $$
 e(V_2)=e(\pgl/U) e(\overline{Y}_2)e(\overline{X}_2)= q^{14}-2q^{13}-7q^{12}+4q^{11}-16q^{10}+84q^9+132q^8-44q^7-65q^6-42q^5-45q^4.
 $$
\item $V_{3}=\lbrace (A_1,B_1,A_2,B_2,A_3,B_3) \mid [A_1,B_1][A_2,B_2]=[B_3,A_3]\sim J_{-} \rbrace$.
 $$
 e(V_3)=e(\pgl/U) e(\overline{Y}_3)e(\overline{X}_3)=q^{14}+3q^{13}-4q^{12}+3q^{11}+54q^{10}+57q^9+84q^8-63q^7-135q^6.
 $$
\item $V_{4}= \lbrace (A_1,B_1,A_2,B_2,A_3,B_3)  \mid [A_1,B_1][A_2,B_2]=[B_3,A_3]\sim \begin{pmatrix} \lambda & 0 \\ 0 & \lambda^{-1} \end{pmatrix}, \lambda \neq 0, \pm 1  \rbrace$.
\end{itemize}

For computing $e(Y_{4})$, we define
$$
\overline{V}_{4,\lambda}:= \left\{ (A_1,B_1,A_2,B_2,A_3,B_3) \mid [A_1,B_1][A_2,B_2]=[B_3,A_3]= \begin{pmatrix} \lambda & 0 \\ 0 & \lambda^{-1} \end{pmatrix} \right\},
$$
for $\lambda \neq 0,\pm 1$. There is a fibration $\overline{V}_{4} \longrightarrow \mathbb{C}-  \lbrace 0,\pm 1 \rbrace$ with fiber $\overline{V}_{4,\lambda}$. 
Note that $\overline{V}_{4,\lambda}\cong \overline{Y}_{4,\lambda} \times \overline{X}_{4,\lambda}$, so
\begin{align*}
 R(\overline{V}_4/\ZZ_2) =&R(\overline{Y}_4/\ZZ_2) \otimes R(\overline{X}_4/\ZZ_2) \\
 =&((q^9-3q^7+6q^5)T-(45q^5+15q^3)S_{2}+(15q^6+45q^4)S_{-2}+(-6q^4+3q^2-1)S_{0}) \\ & \otimes (q^3T-3qS_{2}+3q^2S_{-2}-S_{0}) \\
 =& (q^{12}-3q^{10}+51q^8+270q^6+51q^4-3q^2+1)T +(-3q^{10}-36q^8-66q^6-36q^4-3q^2)S_{2} \\
 & +(3q^{11}+6q^9+63q^7+63q^5+6q^3+3q)S_{-2} + (-q^9-183q^7-183q^5-q^3)S_{0},
\end{align*}
which we write as $R(\overline{V}_{4}/ \ZZ_2)=\tilde{a}T+\tilde{b}S_{2}+\tilde{c}S_{-2}+\tilde{d}S_{0}$. If we apply \eqref{eqn:RX4-RX4Z2->eX4},
\begin{align*}
 e(V_{4}) & =  q(q^{2}-2q-1)\tilde{a}-q(q+1)(\tilde{b}+\tilde{c})-2q\tilde{d} \\
 &  =  q^{15}-2q^{14}-7q^{13}+6q^{12}+51q^{11}-70q^{10}+192q^9-171q^8-216q^7+237q^6-24q^5+5q^4+4q^3-5q^2-q.
\end{align*}
{}From this
 \begin{align} \label{eqn:VV}
 e(V) & =e(V_0)+e(V_1)+e(V_2)+e(V_3)+e(V_4) \nonumber \\
 & = q^{15}-5q^{13}+q^{12}+73q^{11}+9q^{10}+295q^9-5q^8-295q^7-5q^6-73q^5+5q^3-q.
 \end{align}

\subsection{Contribution from reducibles}

To compute the E-polynomial of $\mathcal{M}=\cM(\SL(2,\CC))$, 
we need to take a GIT quotient and differentiate between reducible and irreducible orbits. 

In \cite[Section 8]{lomune}, this analysis is done in the case of $g=2$, by stratifying the set
of irreducible orbits, and computing the E-polynomial of each stratum. The number of strata
increases rapidly with the genus. Therefore, for $g=3$ we are going to 
follow the method in \cite{lamu} which consists on computing the E-polynomial of the
reducible locus (which has fewer strata) and substracting it from the total.

A reducible representation given by $(A_1,B_1,A_2,B_2,A_3,B_3)$ is $S$-equivalent to
\begin{equation}\label{eqn:aa}
\left( \begin{pmatrix} \lambda_{1} & 0 \\ 0 & \lambda_{1}^{-1} \end{pmatrix},  \begin{pmatrix} \lambda_{2} & 0 \\ 0 & \lambda_{2}^{-1} \end{pmatrix},  
\begin{pmatrix} \lambda_{3} & 0 \\ 0 & \lambda_{3}^{-1} \end{pmatrix},  \begin{pmatrix} \lambda_{4} & 0 \\ 0 & \lambda_{4}^{-1} \end{pmatrix},  
\begin{pmatrix} \lambda_{5} & 0 \\ 0 & \lambda_{5}^{-1} \end{pmatrix},  \begin{pmatrix} \lambda_{6} & 0 \\ 0 & \lambda_{6}^{-1}\end{pmatrix} \right),
\end{equation}
under the equivalence relation $(\lambda_{1},\lambda_{2},\lambda_{3},\lambda_{4},\lambda_{5},\lambda_{6}) \sim (\lambda_{1}^{-1},\lambda_{2}^{-1},\lambda_{3}^{-1},\lambda_{4}^{-1},\lambda_{5}^{-1},\lambda_{6}^{-1}) $ given by the permutation of the eigenvectors. Under the action $\lambda \mapsto \lambda^{-1}$ we have that $e(\mathbb{C}^{\ast})^{+}=q$, $e(\mathbb{C}^{\ast})^{-}=-1$, so we obtain
\begin{align*}
e(\cM^{red}) & =  e((\mathbb{C}^{\ast})^{6}/ \mathbb{Z}_{2}) \\ 
& =   (e(\mathbb{C}^{\ast})^{+})^{6}+ \binom{6}{2} (e(\mathbb{C}^{\ast})^{+})^{4}(e(\mathbb{C}^{\ast})^{-})^{2} + \binom{6}{4} (e(\mathbb{C}^{\ast})^{+})^{2}(e(\mathbb{C}^{\ast})^{-})^{4}+ (e(\mathbb{C}^{\ast})^{-})^{6} \\
& =  q^{6}+15q^{4}+15q^{2}+1.
\end{align*}

A reducible representation happens when there is a common eigenvector. So in a suitable basis, it is
\begin{equation}\label{eqn:bb}
\left( \begin{pmatrix} \lambda_{1} & a_1 \\ 0 & \lambda_{1}^{-1} \end{pmatrix},  \begin{pmatrix} \lambda_{2} & a_2 \\ 0 & \lambda_{2}^{-1} \end{pmatrix},  
\begin{pmatrix} \lambda_{3} & a_3 \\ 0 & \lambda_{3}^{-1} \end{pmatrix},  \begin{pmatrix} \lambda_{4} & a_4 \\ 0 & \lambda_{4}^{-1} \end{pmatrix},  
\begin{pmatrix} \lambda_{5} & a_5 \\ 0 & \lambda_{5}^{-1} \end{pmatrix},  \begin{pmatrix} \lambda_{6} & a_6 \\ 0 & \lambda_{6}^{-1}\end{pmatrix} \right).
\end{equation}
This is parametrized by $(\CC^*\x \CC)^6$. The condition $[A_{1},B_{1}][A_{2},B_{2}][A_{3},B_{3}]=\Id$ is rewritten as
\begin{equation}\label{eqn:cc}
 \lambda_2 ( \lambda_1^2- 1)a_2 -\lambda_1( \lambda_2^2- 1)a_1 
 +\lambda_4 ( \lambda_3^2- 1)a_4 -\lambda_3( \lambda_4^2- 1)a_3
 +\lambda_6 ( \lambda_5^2- 1)a_6 -\lambda_5( \lambda_6^2- 1)a_5  =0.
\end{equation}
There are four cases:
\begin{itemize}
\item $R_1$ given by $(a_1,a_2,a_3,a_4,a_5,a_6) \in
\langle (\lambda_1-\lambda_1^{-1},\lambda_2-\lambda_2^{-1},\lambda_3-\lambda_3^{-1},
\lambda_4-\lambda_4^{-1},\lambda_5-\lambda_5^{-1}, \lambda_6-\lambda_6^{-1}) \rangle$ and $(\lambda_{1},\lambda_{2},\lambda_{3},\lambda_{4},\lambda_{5},\lambda_{6}) \neq (\pm 1, \pm 1, \pm 1, \pm 1, \pm 1, \pm 1)$.
Then we can conjugate the representation (\ref{eqn:bb}) to the diagonal form (\ref{eqn:aa}). 
In this case we can suppose all $a_i=0$, and the stabilizer are the diagonal matrices $D\subset \PGL(2,\CC)$.
There is an action of $\ZZ_2$ given by interchanging of the two basis vectors, and if we write 
$A:= (\CC^*)^6-\{(\pm 1,\pm 1, \pm 1, \pm 1, \pm 1, \pm 1 )\}$, the stratum
is $(A \times \PGL(2,\CC)/D)/\ZZ_2$. Note that $e(A)^+= q^{6}+15q^{4}+15q^{2}-63$,
$e(A)^-=e(A)-e(A)^+=-(6q^5+20q^3+6q)$, and
$e(\PGL(2,\CC)/D)^+=q^2$, $e( \PGL(2,\CC)/D)^-=q$. So
  $$e(R_1)= q^8+9q^6-5q^4-69q^2.$$
 
\item  $R_2$ given by $(a_1,a_2,a_3,a_4,a_5,a_6) \not\in
\langle (\lambda_1-\lambda_1^{-1},\lambda_2-\lambda_2^{-1},\lambda_3-\lambda_3^{-1},
\lambda_4-\lambda_4^{-1},\lambda_5-\lambda_5^{-1}, \lambda_6-\lambda_6^{-1}) \rangle$ and $(\lambda_{1},\lambda_{2},\lambda_{3},\lambda_{4},\lambda_{5},\lambda_{6}) \neq (\pm 1, \pm 1, \pm 1, \pm 1, \pm 1, \pm 1)$.
Then (\ref{eqn:cc}) determines a hyperplane $H\subset \CC^6$, and the
condition for $(a_1,a_2,a_3,a_4,a_5,a_6)$ defines 
 a line $\ell\subset H$. If $U'\cong D\x U$ denotes the upper triangular matrices,
we have a surjective map $A \x (H-\ell) \x \PGL(2,\CC) \to R_2$ and the fiber is isomorphic
to $U'$. So
 \begin{align*}
 e(R_2) &=((q-1)^6-64)(q^5-q) (q^3-q)/(q^2-q) \\
 &= q^{12}-5q^{11}+9q^{10}-5q^9-6q^8+14q^7-78q^6-58q^5+5q^4-9q^3+69q^2+63q.
 \end{align*}
 
\item $R_3$, given by $(\lambda_{1},\lambda_{2},\lambda_{3},\lambda_{4},\lambda_{5},\lambda_{6})= (\pm 1, \pm 1, \pm 1, \pm 1, \pm 1, \pm 1)$, $(a_{1},a_{2},a_{3},a_{4},a_{5},a_{6})=(0,0,0,0,0,0)$. This is the case where $A_{i}=B_{i}= \pm \Id$, $i=1,2,3$, which gives $64$ points. Therefore
$$
e(R_3)=64.
$$

\item $R_4$, given by $(\lambda_{1},\lambda_{2},\lambda_{3},\lambda_{4},\lambda_{5},\lambda_{6}) = (\pm 1, \pm 1, \pm 1, \pm 1, \pm 1, \pm 1)$, $(a_{1},a_{2},a_{3},a_{4},a_{5},a_{6})\neq (0,0,0,0,0,0)$. In this case, there is at least a matrix of Jordan type. The diagonal matrices act projectivizing the set $(a_{1},a_{2},a_{3},a_{4},a_{5},a_{6})\in \CC -\{(0,0,0,0,0,0) \}$ and the stabilizer is isomorphic to $U$. So
$$
e(R_4)=64\, e(\mathbb{P}^5)e(\PGL(2,\CC)/U)=64q^7+64q^6-64q-64.
$$
\end{itemize}

Adding all up, we have
\begin{align*}
 e(V^{red}) & = e(R_1)+e(R_2)+e(R_3)+e(R_4)=  q^{12}-5q^{11}+9q^{10}-5q^9-5q^8+78q^7-5q^6-58q^5-9q^3-q,
\end{align*}
and hence
 $$
 e(V^{irr})=e(V)-e(V^{red})= 
 q^{15}-5q^{13}+78q^{11}+300q^9-373q^7-15q^5+14q^3 ,
 $$
and thus
 $$
 e(\cM^{irr})=e(V^{irr})/e(\pgl)= q^{12}-4q^{10}+74q^8+374q^6+q^4-14q^2.
 $$

Finally, we have
$$ 
 e(\cM)=e(\cM^{red})+ e(\cM^{irr})=q^{12}-4q^{10}+74q^8+375q^6+16q^4+q^2+1.
$$

\end{document}